\documentclass[11pt]{article}

\usepackage{hyperref}
\usepackage{amsmath}
\usepackage{amssymb}
\usepackage{amsthm}
\usepackage{latexsym}
\usepackage{color}
\usepackage{mathtools}

\usepackage{graphicx}
\usepackage{appendix}
\usepackage{color} 
\usepackage{enumerate}
\usepackage{diagbox}
\usepackage[T1]{fontenc}
\usepackage[english]{babel}
\usepackage[table]{xcolor}
\usepackage{calrsfs}
\usepackage{geometry}
\usepackage{url}
\DeclareSymbolFont{calletters}{OMS}{cmsy}{m}{n}
\DeclareSymbolFontAlphabet{\mathcal}{calletters}
\DeclareMathAlphabet{\mathpzc}{OT1}{pzc}{m}{it}
\usepackage{enumitem}

\def\be{\begin{eqnarray}}
\def\ee{\end{eqnarray}}

\def\b*{\begin{eqnarray*}}
\def\e*{\end{eqnarray*}}

\newtheorem{Theorem}{Theorem}[part]
\newtheorem{Definition}{Definition}[part]
\newtheorem{Proposition}{Proposition}[part]

\newtheorem{Assumption}{Assumption}[part]
\newtheorem*{Assumptionn}{Assumption}
\newtheorem{Lemma}{Lemma}[part]

\newtheorem{Remark}{Remark}[part]
\newtheorem{Example}{Example}[part]

\makeatletter \@addtoreset{equation}{section}

\@addtoreset{Definition}{section}

\@addtoreset{Theorem}{section}

\@addtoreset{Proposition}{section}

\@addtoreset{Property}{section}

\@addtoreset{Assumption}{section}

\@addtoreset{Corollary}{section}

\@addtoreset{Lemma}{section}

\@addtoreset{Remark}{section}

\@addtoreset{Example}{section}

\@addtoreset{Condition}{section}

\def \E{\mathbb{E}}
\def \F{\mathbb{F}}

\def \P{\mathbb{P}}

\def \R{\mathbb{R}}

\def\Fc{{\cal F}}
\def\Gc{{\cal G}}

\def\Pc{{\cal P}}

\def\einf{{\rm ess \, inf}}
\def\esup{{\rm ess \, sup}}

\def\={\;=\;}
\def\.{\;.}

\def \i{1\!\mbox{\rm I}}
\def\1{{\bf 1}}

\def\einf{{\rm ess \, inf}}
\def\esup{{\rm ess \, sup}}

\def\b*{\begin{eqnarray*}}
\def\e*{\end{eqnarray*}}

 \def\normeL2#1{\left\|{#1}\right\|_{L^2}}
 
 \title{Contract theory in a VUCA world}

 \author{Nicol\'as {\sc Hern\'andez Santib\'anez} \footnote{Department of Mathematics, Univerisity of Michigan, nihernan@umich.edu.} \and Thibaut {\sc Mastrolia}\footnote{CMAP-\'Ecole Polytechnique, Route de Saclay, Palaiseau, France. thibaut.mastrolia@polytechnique.fr}}
 
             \date{\today}

\begin{document}

 \maketitle

\begin{abstract} \noindent In this paper we investigate a Principal-Agent problem with moral hazard under Knightian uncertainty. We extend the seminal framework of Holmstr\"om and Milgrom by combining a Stackelberg equilibrium with a worst-case approach. We investigate a general model in the spirit of \cite{cvitanic2017dynamic}. We show that optimal contracts depend on the output and its quadratic variation, as an extension of the works of \cite{mastrolia2016moral} (by dropping all the restrictive assumptions) and \cite{sung2017optimal} (by considering a general class of admissible contracts). We characterize the best reaction effort of the agent through the solution to a second order BSDE and we show that the value of the problem of the Principal is the viscosity solution of an Hamilton-Jacobi-Bellman-Isaacs equation, without needing a dynamic programming principle, by using stochastic Perron's method. 
\vspace{5mm}

\noindent{\bf Key words:} moral hazard, Principal-Agent, second order BSDEs, volatility uncertainty, Hamilton-Jacobi-Bellman-Isaacs PDEs, stochastic Perron's method.
\vspace{5mm}

\noindent{\bf AMS 2010 subject classifications: 	93E20, 91A15, 49L25.} 

\end{abstract}

\section{Introduction}

Coming from the US army and used afterwards in the business glossary, the acronym VUCA reflects the major issues encountered to investigate risk analysis: \textbf{V}olatility, \textbf{U}ncertainty, \textbf{C}omplexity and \textbf{A}mbiguity\footnote{See for instance \textit{What VUCA Really Means for You,} N. Bennett and GJ. Lemoine, \textit{Harvard Business Review}, January-February 2014.}. The notion of volatility is at the heart of mathematical finance, where unstable properties of financial products, such as prices, are modeled through the presence of noise in their dynamics. Uncertainty is the lack of knowledge for an active agent, due to information asymmetries between him/her and the other parts involved. Complexity holds when several interconnected entities interact, leading to issues whose solutions are not obvious at first sight. Typically, the difficulties appearing in contract theory with moral hazard come from these first three concepts. In the canonical situation, a first entity named the \textit{Principal} (she) designs a monetary contract to hire another entity, named the \textit{Agent} (he) to manage her wealth. The Agent has the possibility to accept or reject the contract proposed by the Principal, so that the Principal must provide sufficiently good incentives to the Agent constrained to ensure that his reservation utility is attained. The Agent thus provides an effort which directly impacts the value of the Principal's wealth. The main difficulty is that the Principal has to design the contract without observing directly the effort provided by her Agent. We identify commonly this situation with a Stackelberg equilibrium between the Principal and the Agent: first the Principal anticipates the best reaction effort of the Agent for any given fixed salary. Then, taking into account the optimal efforts, she maximizes her utility and computes the optimal contract satisfying the reservation utility constraint. This paradigm appeared in the 1970's in discrete time models and has then been reformulated by Holmstr\"om and Milgrom in \cite{holmstrom1987aggregation} in a continuous time version of the problem in which the work of the Agent is to control the drift of a Brownian diffusion. We refer to the monographs \cite{sung2001lectures} and \cite{cvitanic2012contract} for more explanation, general overviews and mathematical treatments of this theory. \vspace{0.5em}

Several extensions of the work of Holmstr\"om and Milgrom have recently surfaced. A first noticeable extension is the study of Sannikov \cite{sannikov2008continuous} by studying a Principal-Agent problem with a retiring random time chosen by the Principal. In particular, Sannikov roughly emphasized that the problem of the Principal has to be seen as a stochastic control problem where the continuation utility of the Agent is a state variable. This idea was rigorously extended later in the works of Cvitani\'c, Possama\"i and Touzi in \cite{cvitanic2014moral,cvitanic2017dynamic} by investigating a Principal-Agent problem in which the Agent can control both the drift and the volatility of the wealth of the Principal. More precisely, they show that when the Agent takes also supremum over the possible volatilities, his value function is the solution to a second order BSDE (2BSDE for short), which the theory was introduced by Sonner, Touzi and Zhang in \cite{soner2012wellposedness} and improved by Possama\"i, Tan and Zhou in \cite{possamai2015stochastic}. The so-called dynamic programming approach of Cvitani\'c, Possama\"i and Touzi consists in restricting the set of contracts offered to the Agent to a suitable class so that the problem of the Principal is reduced to a standard stochastic control problem associated with a Hamilton-Jacobi-Bellman equation. The main difference between the unrestricted and the restricted class of contracts lies in the absolutely continuity of the increasing process appearing in the 2BSDE associated with the problem of the Agent. By building absolutely continuous approximations of the increasing process, they show that the restricted and the unrestricted problems have the same value.
 \vspace{0.5em} \noindent

In this paper we incorporate the last component of VUCA, \textbf{A}mbiguity, into the standard Principal-Agent problem. To quote Bennett and Lemoine "Ambiguity characterizes situations where there is doubt about the nature of cause-and-effect relationships". We model ambiguity by introducing a third player in the system, named the \textit{Nature}, which randomly modifies the volatility of the project. As usual, the Agent is hired by the Principal to control the drift of an output process and the Principal cannot observe the actions of the Agent. However, the Principal and the Agent are not informed about the volatility of the project and they just have some beliefs about it. Since we work under weak formulation, the uncertainty on the volatility is represented by assigning to the Principal and the Agent different sets of probability measures under which they make their decisions. We adopt a worst case approach against this scenario so that both individuals present an extreme ambiguity aversion to the problem. They act as if the third individual, the Nature, was playing against them and choosing the worst possible volatility. As a consequence, both the Principal and the Agent play zero-sum stochastic differential games against the Nature.
This work is an extension of the models proposed by Mastrolia and Possama\"{i} \cite{mastrolia2016moral} and Sung \cite{sung2017optimal} to more general frameworks, by dropping several explicit and implicit assumptions made in these papers. We consider a more general framework by not considering only exponential utilities and by not restricting \textit{a priori} the class of admissible contracts. \vspace{0.5em}

\noindent Since the seminal work of Isaacs \cite{isaacs1965differential}, differential games and more particularly zero-sum games have received a growing interest. As an overview of works related to this theory, let us recall some noticeable and relevant studies inspiring the present paper and the mathematical tools that they used. Lions and Souganidis have investigated stochastic differential games in \cite{lions1985differential} by using the viscosity solutions theory, introduced in the works of Lions \cite{lions1988viscosity,lions1989viscosity,lions1989viscosity2}. Hamad\`ene and Lepeltier have then proved in \cite{hamadene1995backward} the existence of a saddle point when the so-called Isaacs' condition is satisfied with the help of classical BSDEs. These works were then generalized to more general dynamics by Buckdahn and Li in \cite{buckdahn2008stochastic}, allowing both the drift and the diffusion terms of the output process to be impacted by control processes. Cardaliaguet and Rainer have then introduced a new notion of strategies, called path-wise strategies, to solve a differential game in \cite{cardaliaguet2009stochastic, cardaliaguet2013pathwise}. It was then extended by Bayraktar and Yao in \cite{bayraktar2013yao} by considering unbounded controls and by using a weak dynamic programming approach. Notice that all of these works mainly deal with Isaacs' condition. Buckdahn, Li and Quincampoix have then sucessfully  characterized the value of a stochastic differential games without assuming Isaacs' condition in \cite{buckdahn2014value}, by using viscosity solutions theory together with a randomization procedure of the stochastic control processes. All these works frame zero-sum games under strong formulation. The work of Pham and Zhang \cite{pham2014two} investigates a non-Markovian zero-sum game in the weak formulation, more suitable\footnote{See \cite[Section 10.4.1]{cvitanic2012contract}} for Principal-Agent problems with moral hazard, by using path-dependent PDEs.\vspace{0.5em} 

\noindent All the previous papers treat stochastic differential games with a dynamic programming principle approach (DPP for short). Although El Karoui and Tan have proved in \cite{karoui2013capacities, karoui2013capacities2} that a DPP holds for very general stochastic control problems, this is not the case for stochastic differential games. As explained in the paper of Hamad\`ene, Lepeltier and Peng \cite{hamadene1997bsdes}, a game of type "control against control" may not lead to a DPP. A central point of our work is to avoid DPP by following the stochastic Perron's method developed by Bayraktar and S\^irbu in \cite{bayraktar2012stochastic,bayraktar2013stochastic,bayraktar2014dynkin}, and applied to stochastic differential games by S\^irbu in \cite{sirbu2014stochastic}.  

\vspace{0.5em} \noindent
In our problem, we aim at mixing zero-sum differential games with Stackelberg equilibrium by following partially the dynamic programming approach introduced in \cite{cvitanic2017dynamic}. The Agent does not choose the volatility of the outcome process but his worst case approach leads to reduce his problem to the solution to a 2BSDE, seen as an infimum of BSDEs over a set of probability measures. Unlike \cite{cvitanic2017dynamic}, the problem of the Principal becomes a non-standard stochastic differential game, because the worst probability measures for the Agent and the Principal do not necessarily coincide. The main contribution of this paper is to prove that the value function of the Principal is a \textit{viscosity} solution to the Hamilton-Jacobi-Bellman-Isaacs (HJBI for short) equation associated to a restricted problem, as soon as a comparison result holds. The method that we use is based on the stochastic Perron's method of Bayraktar and S\^{i}rbu \cite{bayraktar2012stochastic,bayraktar2013stochastic,bayraktar2014dynkin,sirbu2014stochastic} and prevents a dynamic programming principle for our problem, which may be quite hard to check in practice. The stochastic Perron's method amounts to a verification result and it consists in proving that the value function of the Principal lies between a viscosity super-solution (the supremum of the stochastic sub-solutions) and a viscosity sub-solution (the infimum of the stochastic super-solutions) of the HJBI equation. Thus, as soon as a comparison theorem holds for such PDE, and the set of stochastic semi-solutions are non-empty, it follows that the value function of the Principal coincides with the unique viscosity solution. Moreover, the DPP also follows from the definition of the stochastic semi-solutions. The only restriction that we made in our work is to deal with piecewise controls for the problem of the Principal. Although this assumption is restrictive, it is very common in stochastic control theory and meaningful as explained in \cite{sirbu2014stochastic}. Moreover, in view of \cite{pham2014two, soner2013dual} we expect that the value function associated to this restricted problem coincide with the general problem. 
\vspace{0.5em} \noindent

The structure of the paper is the following, in Section \ref{preliminaries} we define the framework and the model. The problem of the Agent is solved in Section \ref{section:Agent-problem}. Section \ref{section:principal} is at the heart of our study and is the main contribution of our paper. After having studied the degeneracies of our problem, we prove that the value function of the Principal is a viscosity solution to the HJBI equation associated to a restricted problem, as soon as a comparison result holds, without assuming that a dynamic programming principle holds by using stochastic Perron's method. We also discuss examples in which we can expect that the comparison result is satisfied. To ease the reading of the paper, some technical definitions and the proofs of the two main results are postponed to the appendix.

\section{The model\label{preliminaries}}

\subsection{Canonical process, semi-martingale measure and quadratic variation}

We fix a maturity $T>0$ and a positive integer $d$. Let $C\left( \left[ 0,T \right], \mathbb{R}^d \right)$ be the space of continuous maps from $[0,T]$ into $\mathbb R^d$ and let $ \Omega \mathrel{\mathop:}= \lbrace \omega \in C\left( \left[ 0,T \right], \mathbb{R}^d \right) : \omega_0 = 0 \rbrace $ be the canonical space endowed with the uniform norm 
$$
\| \omega\|_{\infty}=\sup_{t\in [0,T]} \|\omega_t\|.
$$ We denote by $X$ the canonical process on $\Omega$, \textit{i.e.} $ X_t(x) = x_t$, for all $x\in\Omega$ and $t\in[0,T]$. We set $\mathbb G:=(\mathcal G_t)_{t\in [0,T]}$ the filtration generated by $X$ and $\mathbb G^+ := (\mathcal G_{t}^+)_{t \in [0,T]}$ its right limit, where $\mathcal G_{t}^+ :=\bigcap_{s>t} \mathcal G_s$ for $s\in[0,T)$ and $\mathcal G_{T}^+:=\mathcal G_T$. We denote by $\mathbb P_0$ the Wiener measure on $(\Omega,\mathcal G_T)$. Let $\mathbf{M}(\Omega)$ be the set of all probability measures on $(\Omega,\Gc_T)$. Recall the so--called universal filtration $\mathbb{G}^\star :=\lbrace \mathcal{G}^\star _t \rbrace_{0\le t \le T} $ defined as follows
$$
\mathcal{G}^\star _t
:=
\underset{\mathbb{P}\in\mathbf{M}(\Omega)}
{\bigcap}\mathcal{G}_t^{\mathbb{P}},
$$ 
where $\mathcal{G}_t^{\mathbb{P}}$ is the usual completion under $\mathbb{P}$.
 \vspace{0.5em}
 
For any subset $\mathcal{P}\subset\mathbf{M}(\Omega)$, a $\mathcal{P}-$polar set is a $\mathbb{P}-$negligible set for all $\mathbb{P}\in \mathcal{P}$, and we say that a property holds $\mathcal{P}-$quasi-surely if it holds outside some $\mathcal{P}-$polar set. We also introduce the filtration $\mathbb{F}^{\mathcal{P}}:=\lbrace \mathcal{F}^\mathcal{P}_t \rbrace_{0\le t \le T} $, defined by 
$$\mathcal{F}^{\mathcal{P}}_t:=\mathcal{G}_{t}^\star \vee \mathcal{T}^\mathcal{P},\ t\leq T,$$
where $\mathcal{T}^\mathcal{P}$ is the collection of $\mathcal{P}-$polar sets, and its right-continuous limit, denoted $\F^{\Pc,+}:=(\mathcal F_t^{\mathcal P,+})_{t\in [0,T]}$, and we omit the indexation with respect to $\mathcal P$ when there is no ambiguity on it. 
\vspace{0.3em}

For any subset $\mathcal P\subset\mathbf{M}(\Omega)$ and any $(t,\mathbb P)\in [0,T]\times \mathcal P$ we denote
$$
\mathcal P[\mathbb P,\mathbb F^+,t]:=\left\{ \mathbb P'\in \mathcal P,\; \mathbb P'=\mathbb P \text{ on }\mathcal F_{t}^+\right\}.
$$ 

We also recall that for every probability measure $\P$ on $\Omega$ and $\F-$stopping time $\tau$ taking values in $[0,T]$, there exists a family of regular conditional probability distribution (r.c.p.d. for short) $(\P^{\tau}_{ x })_{ x  \in \Omega}$ (see e.g. \cite{stroock2007multidimensional}), satisfying Properties $(i)-(iv)$ of \cite{possamai2015stochastic} and we refer to \cite[Section 2.1.3]{possamai2015stochastic} for more details on it.

\vspace{.5em}

We say that $\mathbb P\in \mathbf{M}(\Omega)$ is a semi--martingale measure if $X$ is a semi--martingale under $\mathbb{P}$. We denote by $\mathcal{P}_S$ the set of all semi-martingale measures. We set $\mathcal M_{d,n}(\R)$ the space of matrices with $d$ rows and $n$ columns with real entries. It is well-known, see for instance the result of \cite{karandikar1995pathwise}, that there exists an $\mathbb F$-progressively measurable process denoted by $\langle X\rangle:=(\langle X\rangle_t)_{t\in [0,T]}$ coinciding with the quadratic variation of $X$, $\mathbb P-a.s.$ for any $\mathbb P\in\mathcal{P}_S$, with density with respect to the Lebesgue measure at time $t\in [0,T]$ denoted by a non-negative symmetric matrix $\widehat{\sigma}_t\in\mathcal M_{d,d}(\R)$ defined by
$$
\widehat{\sigma}_t:= \underset{\varepsilon>0}{\underset{\varepsilon \longrightarrow 0}{\text{lim sup}}}\,  \dfrac{\langle X\rangle_t- \langle X\rangle_{t-\varepsilon}}{\varepsilon}.
$$
The formal definition of all the functional spaces mentioned in this paper can be found in Appendix \ref{appendix:spaces}.

\subsection{Weak formulation of the output process}

We start by defining $\mathfrak A$ and $\mathfrak N$ as the sets of  $\mathbb F$-adapted processes taking values in $A$ and $N$ respectively, where $A,N$ are compact subsets of some finite dimensional space. We call \textit{control process} every pair $(\alpha,\nu)\in\mathfrak A\times\mathfrak N$. To clarify the notations for the rest of the paper, $\alpha$ has to be understood as the control of the Agent and $\nu$ as the control of the Nature. Consider next the volatility coefficient for the controlled process
$$
\sigma: [0,T]\times\Omega\times N\longrightarrow\mathcal M_{d,n}(\mathbb R),
$$ 
which is assumed to be uniformly bounded and such that $\sigma\sigma^\top(\cdot,\mathfrak n)$ is an invertible $\mathbb F$-progressively measurable process for any $\mathfrak n\in N$. For every $(t,x)\in [0,T]\times \Omega $ and $\nu\in\mathfrak{N}$, we set the following SDE driven by an $n$-dimensional Brownian motion $W$
\begin{align}\label{SDE:output}
 X_s^{t,x,\nu}&=x(t)+\int_t^s \sigma(r, X^{t,x,\nu},\nu_r) dW_r,\, s\in[t,T],\\
\nonumber X_r^{t,x,\nu}&=x(r),\, r\in [0,t].
\end{align}

Similarly to \cite{cvitanic2017dynamic}, we build a control model through the weak solutions of SDE \eqref{SDE:output}. We say $(\mathbb P,\nu)$ is a weak solution of \eqref{SDE:output} if the law of $X^{t,x,\nu}_t$ under $\mathbb P$ is $\delta_{x(t)}$\footnote{$\delta$ denotes the Dirac measure.} and there exists a $\mathbb P-$Brownian motion\footnote{We refer to \cite[Theorem 4.5.2]{stroock2007multidimensional} for more details on it.}, denoted by $W^\mathbb P$, such that
\begin{equation}\label{output:decomposition}
X_s=x(t)+\int_t^s \sigma(r,X,\nu_r) dW^\P_r, \, s\in [t,T], \, \mathbb P-a.s.
\end{equation}

We will denote by $\mathcal N(t,x)$ the set of weak solutions to SDE \eqref{SDE:output}. We also define the set $\mathcal P(t,x)$ of probability measures which are components of weak solutions by
$$ \mathcal P(t,x):= \bigcup_{\nu\in\mathfrak N}\mathcal P^\nu(t,x), ~ \text{where }  \mathcal P^\nu(t,x):= \left\{\mathbb P\in \mathbf M(\Omega),\; (\mathbb P,\nu)\in \mathcal N(t,x) \right\}.$$

We conclude this section by showing that the set $\mathcal P(t,x)$ satisfies an important property which is essential to deal with the wellposedness of 2BSDEs, the main tool we will use later to solve the problem of the Agent. We recall the definition of a saturated set of probability measures (see \cite[Definition 5.1]{possamai2015stochastic}).

\begin{Definition}[Saturated set of probability measures.]
A set $\mathcal P\subset \mathbf M(\Omega)$ is said to be saturated if for an arbitrary $\mathbb P\in \mathcal P$, any probability $\mathbb Q\in \mathbf M(\Omega)$ which is equivalent to $\mathbb P$ and under which $X$ is a local martingale, belongs to $\mathcal P$.
\end{Definition}

We thus have the following Lemma, whose proof follows the same lines that \cite[Proof of Proposition 5.3, step (i)]{cvitanic2017dynamic}
\begin{Lemma}
The family $\{\mathcal P(t,x),\, (t,x)\in [0,T]\times \Omega \}$ is saturated.
\end{Lemma}

\subsection{Estimate sets of volatility}\label{section:setprobability}
The beliefs of the Agent and the Principal about the volatility of the project will be summed up in the families of measures $(\mathcal{P}_A(t,x))_{(t,x)\in [0,T]\times\Omega}$ and $(\mathcal{P}_P(t,x))_{(t,x)\in [0,T]\times\Omega}$ respectively, which satisfy that $\mathcal{P}_A(t,x)\cup\mathcal{P}_P(t,x)\subset \mathcal{P}(t,x)$ for every $(t,x)\in[0,T]\times\Omega$. We emphasize that the families $\mathcal P_A$ and $\mathcal P_P$ cannot be chosen completely arbitrarily, and have to satisfy a certain number of stability and measurability properties, which are classical in stochastic control theory, in order to use the theory of 2BSDEs developed in \cite{possamai2015stochastic}. The following assumption guarantees the well--posedness of 2BSDEs defined in the set of beliefs of the Principal and the Agent.

\begin{Assumption}\label{assumption:proba}
For $\Psi=A,P$, the set  $\Pc_\Psi(t, x)$ satisfies properties $(iii)$ (semi-analyticity of the graph of $\mathcal P_\Psi$), $(iv)$ (stability under conditioning) and $(v)$ (stability under concatenation) in \cite[Assumption 2.1]{possamai2015stochastic}.
\end{Assumption}
In particular, property $(iii)$ implies that the sets $\mathcal{P}_A(t,x)$ and $\mathcal{P}_P(t,x)$ at time $t=0$ are independent of $x$. We thus define
$$
\mathcal{P}_A:=\mathcal{P}_A(0,x),~ \mathcal{P}_P:=\mathcal{P}_P(0,x) ~ \textrm{ for every } x\in\Omega.
$$

\vspace{0.5em} \noindent
An example of estimate sets of volatility which satisfy Assumption \ref{assumption:proba} is the learning model presented in \cite{mastrolia2016moral}.

\begin{Example}
Consider, for $\Psi=A,P$, set--valued  processes $\mathbf{D}_\Psi:[0,T]\times\Omega\longmapsto 2^{\R_+^\star}$ such that for every $t\in[0,T]$
$$
\left\{(s,\omega,A)\in[0,t]\times\Omega\times\R_+^\star,\ A\in \mathbf{D}_\Psi(s,\omega)\right\}\in\mathcal B([0,t])\otimes\Fc_t\otimes\mathcal B(\R_+^\star),
$$
where $\mathcal B([0,t])$ and $\mathcal B(\R_+^\star)$ denote the Borel $\sigma-$algebra of $[0,t]$ and $\R_+^\star$ respectively. Define next, for every $(t,w)\in[0,T]\times\Omega$, the set $\mathcal{P}_\Psi(t,\omega)$ as the set of probability measures $\P\in\mathbf{M}(\Omega)$ such that
$$\widehat \sigma_s(w')\in\mathbf{D}_\Psi(s+t,\omega\otimes_t w'),\; \text{for $ds\otimes d\P-a.e.$ $(s,w')\in[0,T-t]\times\Omega$}.$$
It is shown in \cite{nutz2013constructing} that the sets $\mathbf{P}_\Psi(t,\omega)$ satisfy Assumption \ref{assumption:proba}. 
\end{Example}

\vspace{0.5em} \noindent
In the context of the previous example, \cite{mastrolia2016moral} studies the case where 
$$\mathbf{D}_A(t,\omega)=[\underline\sigma^A_t(\omega),\overline\sigma^A_t(\omega)],\; \mathbf{D}_P(t,\omega)=[\underline\sigma^P_t(\omega),\overline\sigma^P_t(\omega)],$$
for certain processes $(\underline \sigma^P,\underline\sigma^A,\overline\sigma^P,\overline\sigma^A)\in\left(\mathbb H^0(\mathbb R_+^*,\mathbb F )\right)^4$. We refer to their paper for an interpretation of such model. 

\vspace{0.5em} \noindent
To conclude this section, we define the set of weak solutions to the SDE \eqref{SDE:output} associated to the beliefs of the Principal of the Agent
$$ \mathcal N_A(t,x) = \left\{ (\mathbb P,\nu)\in \mathcal N(t,x): \mathbb{P} \in \mathcal{P}_A(t,x) \right\},~ \mathcal N_P(t,x) = \left\{ (\mathbb P,\nu)\in \mathcal N(t,x): \mathbb{P} \in \mathcal{P}_P(t,x) \right\}. $$
We define the sets $\mathcal{N}_A$ and $\mathcal{N}_P$ equivalently. The importance of these sets is that, as explained in the next section, both the Principal and the Agent consider that the volatility of the outcome process is chosen from one of them, according to their beliefs.

\subsection{The contracting problem}

We study a generalization of both the classical problem of Holmstr\"om and Milgrom \cite{holmstrom1987aggregation} and the problem of moral hazard under volatility uncertainty studied in \cite{mastrolia2016moral, sung2017optimal}. In our model, the Agent is hired by the Principal to control the drift of the outcome process $X$, but none of them have certainty about what is the volatility of the project. Both sides observe $X$ and have a "worst-case" approach to the contract, in the sense that they act as if a third player, the "Nature", was playing against them by choosing the worst possible volatility.

\subsubsection{Admissible efforts}
As usual in the literature, we work under the weak formulation of the Principal-Agent problem. Therefor, the set of controls of the Agent is restricted to the ones for which an appropriate change of measure can be applied to the weak solutions of SDE \eqref{SDE:output}. In this section we precise the condition required on a control to be an admissible effort and the impact of the actions of the Agent in the outcome process.

\vspace{0.5cm} \noindent
The Agent exerts an effort $\alpha\in\mathfrak{A}$ to manage the project, unobservable by the Principal, impacting the outcome process through the drift coefficient $ b: [0,T]\times\Omega\times A\times N\longrightarrow\R^n$, which satisfies that $b(\cdot,a,\mathfrak n)$ is an $\mathbb F$-progressively measurable process for every $(a,\mathfrak n)\in A\times N$. The actions of the Agent are costly for him, so his benefits are penalized by a cost function $c:[0,T]\times \Omega\times A\longrightarrow\R$ such that for every $a\in A$, $c(\cdot,a)$ is an $\mathbb F-$progressively measurable process. We assume that for some $p>1$ there exists $\kappa\in(1,p]$ such that
\begin{equation}
\sup_{\mathbb P\in \mathcal P_A} \mathbb E^\mathbb P\left[ \underset{0\leq t\leq T}{\esup^{\mathbb P}} \, \mathbb E^\mathbb{P} \left[\left( \int_0^T \sup_{a\in A} | c(s,X,a)|^\kappa ds \right)^\frac{p}{\kappa} \Big| \mathcal G_{t}^+ \right] \right]<+\infty.
\end{equation}
The Agent discounts the future through a map $k:[0,T]\times \Omega\times A\times N\longrightarrow\R$, such that $k(\cdot,a,\mathfrak n)$ is an $\mathbb F-$progressively measurable process for every $(a,\mathfrak n)\in A\times N$. For some $(\ell,m,\underline m)\in[1,+\infty)\times [\ell,+\infty)\times (0,\ell+m-1]$, we impose the following conditions on the maps $b$, $c$ and $k$

\begin{Assumptionn}[$\mathbf H^{\ell,m,\underline m}$] There exists $0< \underline \kappa<\kappa$ such that for any $(t,x,a,\mathfrak n)\in [0,T]\times \Omega\times A\times N$
\begin{enumerate}[label=(\roman*)]
\item The drift $b$ satisfies
$$ \|b(t,x,a,\mathfrak n)\|\leq \kappa\left( 1 +\|x\|_{t,+\infty}+\|a\|^\ell\right),~ \|\partial_ab(t,x,a,\mathfrak n)\|\leq \kappa\left( 1+\|a\|^{\ell-1}\right).$$ 
\item The map $a\longmapsto c(t,x,a)$ is increasing, strictly convex and continuously differentiable for any $(t,x)\in [0,T]\times\Omega$ and satisfies
$$0\leq c(t,x,a)\leq \kappa \Bigg(1+\|x\|_{t,\infty}+\|a\|^{\ell+m}\Bigg), $$
$$\underline\kappa  \|a\|^{\underline m} \leq  \|\partial_ac(t,x,a)\|\leq \kappa  \Bigg(1+\|a\|^{\ell+m-1}\Bigg) \, \text{ and } \overline{\lim}_{\|a\|\to \infty} \frac{c(t,x,a)}{\|a\|^\ell}=+\infty.$$
\item The discount factor $k$ is uniformly bounded by $\kappa$.
\end{enumerate}
\end{Assumptionn}

\begin{Remark}
For $(\ell,m,\underline m)=(1,1,1)$ we exactly recover the model studied in \cite{mastrolia2016moral}.
\end{Remark}

We present finally the definition of admissible efforts of the Agent.

\begin{Definition}[Admissible efforts]\label{def:admissiblecontrol}
A control process $\alpha\in\mathfrak{A}$ is said to be admissible, if for every $(\mathbb P,\nu) \in \mathcal N_A$ the following process is an $(\mathbb F,\mathbb P)$--martingale
\begin{equation}\label{alphamartingale}
\left(\mathcal E\left(\int_0^t   \sigma^\top(\sigma\sigma^\top)^{-1}(s,X,\nu_s) b(s,X,\alpha_s,\nu_s) \cdot dW_s^{\mathbb P}\right)\right)_{t\in[0,T]}.
\end{equation}
We denote by $\mathcal A$ the set of admissible efforts.
\end{Definition}

Finally, we present the impact of the actions of the Agent in the outcome process. Consider an admissible effort $\alpha\in \mathcal A$ and $(t,x)\in[0,T]\times\Omega$. For every subset $\mathcal N\subset \mathcal N(t,x)$ define 
$$ 
\mathcal N^{\alpha}:=\left\{ (\mathbb P^\alpha,\nu), \frac{d\mathbb P^\alpha}{d\mathbb P}=\mathcal E\left(\int_t^T  \sigma^\top(\sigma\sigma^\top)^{-1}(s,X,\nu_s)b(s,X,\alpha_s,\nu_s) \cdot dW_s^{\mathbb P}\right), (\mathbb P,\nu)\in \mathcal N\right\} .
$$ 
Thus, under Assumption $(\mathbf H^{\ell,m,\underline m})$, by Girsanov's Theorem we have for any $\alpha\in\mathcal A$, and for any $(\mathbb P^\alpha,\nu)\in  \mathcal N^{\alpha}$

\begin{equation}\label{output:decomposasitiondrift}
X_s=x_t+\int_t^s b(r,X,\alpha_r, \nu_r) dr+\int_t^s \sigma(r,X,\nu_r) dW^{\alpha}_r, \, s\in [t,T], \, \mathbb P^\alpha-a.s.,
\end{equation}
where $W^\alpha$ is a $\mathbb P^\alpha-$Brownian motion. More precisely, 
$$W^{\alpha}:=W^\mathbb P-\int_t^\cdot  \sigma^\top(\sigma\sigma^\top)^{-1}(r,X,\nu_r) b(r,X,\alpha_r, \nu_r) dr,$$
for some $\mathbb P\in \mathcal P$. 

\subsubsection{Admissible contracts}

The Principal offers to the Agent a final salary taking place on the horizon $T$. Since the Principal can observe merely the outcome process $X$, a contract corresponds to an $\mathcal F_T$-measurable random variable $\xi$. The Agent benefits from the payments of the Principal through his utility function $U_A:\R\longrightarrow\R$, which depends on his terminal remuneration and is a continuous, increasing and concave map. The Principal benefits from her wealth, penalized by the salary given to the Agent through her utility function $U_P:\mathbb R\longrightarrow \R$ which is a continuous, increasing and concave map. The outcome process is not necessarily monetary so the Principal possesses a liquidation function $L:\R\longrightarrow\R$ which is assumed to be continuous with linear growth. The following (classical) notion of admissibility for the set of contracts proposed by the Principal is due to the fact that we will reduce later the problem of the Agent to solve a 2BSDE.

\begin{Definition}[Admissible contracts]\label{def:admissible}
A contract $\xi$ is called admissible, if 
\begin{itemize}
\item For some $p> 1$ there exists $\kappa\in [1,p)$ such that $U_A(\xi)\in \mathbb L^{p,\kappa}_{0}(\mathbb F, \mathcal{P}_A)$.
\item For any $(\mathbb P,\nu)\in \mathcal N_P$ we have $\mathbb E^\mathbb P\left[U_P\left(L(X_T)-\xi\right) \right]<+\infty.$
\end{itemize}
We denote by $\mathfrak{C}$ the class of admissible contracts.
\end{Definition}

\subsubsection{The problem of the Agent}

For a given contract $\xi\in\mathfrak C$ offered by the Principal, the utility of the Agent at time $t=0$, if he performs the action $\alpha\in\mathcal A$, is given by his worst--case approach over the set $\mathcal{N}_A^\alpha$ of weak solutions to \eqref{SDE:output} associated to his beliefs. That is
$$
u^A_0(\xi,\alpha):=\underset{(\P,\nu)\in\mathcal N^\alpha_A}{\inf}\E^\P\left[ \mathcal K^{\alpha,\nu}_{0,T} U_A(\xi) - \int_0^T \mathcal K^{\alpha,\nu}_{0,s}c(s,X,\alpha_s) ds\right],
$$
where
$$\mathcal K^{\alpha,\nu}_{s,t}:= \exp\left( -\int_s^t k(u,X,\alpha_u,\nu_u) du\right),\; 0\leq s\leq t\leq T.$$
The problem of the Agent, consisting into finding the action which maximizes his utility, is therefore
\begin{equation}\label{problem:Agent}
U_0^A(\xi):=\sup_{\alpha \in \mathcal A}\underset{(\P,\nu)\in\mathcal N^\alpha_A }{\inf}\E^\P\left[ \mathcal K^{\alpha,\nu}_{0,T} U_A(\xi) - \int_0^T \mathcal K^{\alpha,\nu}_{0,s} c(s, X,\alpha_s)ds\right].
\end{equation}

We will denote by $\mathcal A^\star(\xi)$ the set of optimal $\alpha \in  \mathcal A$ when $\xi$ is offered, and define the set of optimal weak solutions 
$$\mathcal N_A^\star(\xi):= \bigcup_{\alpha^\star\in \mathcal A^\star(\xi)} \mathcal N_A^{\alpha^\star}.$$

\subsubsection{The problem of the Principal}

Since the strategy of the Principal is to anticipate the response of the Agent to the offered contracts, she is restricted to offer contracts such that the Agent can optimally choose his Actions. Moreover, the Agent accepts only contracts under which he obtains more benefits than his reservation utility $R_0$. Therefore, the set of admissible contracts is restricted to 
$$\Xi:=\{\xi\in \mathfrak C,\, \mathcal A^\star(\xi)\neq \emptyset,\, U_0^A(\xi)\geq R_0 \}.$$

Notice that for any $\xi\in \Xi$, the set $ \mathcal A^\star(\xi)$ is not necessarily reduced to a singleton. As is common in the literature, we will assume that when there is more than one optimal strategy for the Agent, he chooses one which is best for the Principal. We denote such a strategy by $\alpha^\star(x,\xi)$. Thus, the problem of the Principal is to find the contract which maximizes her worst--case utility (under her own beliefs)
\begin{equation}\label{pb:Principal}
U_0^p:=\sup_{\xi\in \Xi}  \underset{(\P,\nu)\in\mathcal N_P^{\alpha^\star(x,\xi)} }{\inf} \mathbb E^\mathbb P\left[U_P\left(L(X_T)-\xi\right) \right].
\end{equation}

\begin{Remark}
For the sake of simplicity, we do not add any discount factor for the Principal's problem \eqref{pb:Principal}. A model dealing with a discount factor $k^P:[0,T]\times \Omega\longrightarrow \mathbb R$ could be easily studied and does not add any difficulties, as soon as $k^P$ is sufficiently integrable, by modifying the HJBI equation \eqref{HJBI} below. 
\end{Remark}

\section{Solving the Agent problem \textit{via} 2BSDE} \label{section:Agent-problem}
In this section we study the Agent's problem \eqref{problem:Agent}. We follow both the study made in Section 4.1 of \cite{mastrolia2016moral} by extending it to a more general framework, and \cite{cvitanic2017dynamic} by adding uncertainty on the volatility. We mention also that another approach which does not use the theory of 2BSDEs has been proposed in \cite{sung2017optimal}.

\subsection{Definition of the Hamiltonian}
Define the function $F:[0,T]\times \Omega\times \R\times \R^d\times A\times N\longrightarrow \R$ by
$$
F(t,x,y,z,a,\mathfrak n):= -k(t, x, a,\mathfrak n)y-c(t,x,a)+ b(t,x,a,\mathfrak n)\cdot z.
$$

Define also for every $(t,x,\Sigma)\in [0,T]\times\Omega\times \mathcal{S}_d^+ $  the set
$$
V_t(x,\Sigma):=\left\{ \mathfrak n\in N, \sigma(t,x,\mathfrak n)\sigma^\top(t,x,\mathfrak n) = \Sigma \right\},
$$ and denote by $\mathcal{V}(\hat{\sigma}^2)$ the set of controls $\nu\in \mathfrak N$ with values in $V_t(x,\hat{\sigma}_t^2)$, $dt\otimes\mathbb{P}$-a.e. for every $\mathbb P\in \mathcal P_A$.

\vspace{.5em}
The Hamiltonian $H:[0,T]\times \Omega\times \R\times \R^d\times \mathcal S_d^+\longrightarrow \R$ associated with the problem of the Agent \eqref{problem:Agent} is defined by (see \cite{buckdahn2008stochastic}) 
$$ 
H(t,x,y,z,\gamma):=\inf_{\Sigma\in \mathcal S_d^+}\left\{\frac12 \text{Tr}(\Sigma\gamma)+  \inf_{\mathfrak n\in V_t(x,\Sigma)} \sup_{a\in  A}  F(t,x,y,z,a,\mathfrak n) \right\}.
$$

Notice that the infimum with respect to $\mathfrak n\in N$ in the Hamiltonian has been taken in two stages with the introduction of the sets $V_t(x,\Sigma)$.  We assume that the following assertion is enforced.
\begin{Assumption} \label{assumption:isaacs}
The following Isaac's condition is satisfied for any $(t,x,y,z,\Sigma)\in [0,T]\times \Omega\times  \mathbb R^{d+1}\times \mathcal{S}_d^+$
\begin{equation}\label{isaac condition}  
\inf_{\mathfrak n\in V_t(x,\Sigma)}\sup_{a\in A} F(t,x,y,z,a,\mathfrak n) = \sup_{a\in A} \inf_{\mathfrak n\in V_t(x,\Sigma)}F(t,x,y,z,a,\mathfrak n).
\end{equation}
\end{Assumption}

Let us define the map $F^\star:[0,T]\times \Omega\times  \mathbb{R}^{d+1} \times \mathcal{S}_d^+ \longrightarrow \R$ by
$$
F^\star (t,x,y,z,\Sigma)  :=  \sup_{a\in A} \inf_{\mathfrak n\in V_t(x,\Sigma)} F(t,x,y,z,a,\mathfrak n).
$$
We thus state a fundamental lemma on the growth of any control $\alpha^\star$ which is a saddle point in \eqref{isaac condition}. We refer to the proof of \cite[Lemma 4.1]{emp16} which fits our setting. 

\begin{Lemma}\label{lemmagrowth} Let Assumption $(\mathbf H^{\ell,m,\underline m})$ hold. Then, for any $(t,x,y,z,\Sigma)\in [0,T]\times \Omega\times  \mathbb R^{d+1}\times \mathcal{S}_d^+$ and for any maximiser $\alpha^\star$  of $F^\star(t,x,y,z,\Sigma)$, there exists some positive constant $C$ such that
$$\|\alpha^\star(t,x,y,z,\Sigma )\|\leq C\left(1+ \|z\|^{\frac{1}{\underline m+1-\ell}} \right),  $$
$$ |F^\star(t,x,y,z,\Sigma)| \leq C\left(1 +\|x\|_{t,\infty}+|y| +\|z\|^{\frac{\ell+m}{\underline m+1-\ell}}\right).$$
\end{Lemma}

\subsection{2BSDEs representation of the Agent's problem}

Consider the following 2BSDE
\begin{equation}\label{eq:2bsde}Y_t=U_A(\xi)+\int_t^T F^\star(s,X,Y_s,Z_s,\hat{\sigma}_s^2)ds-\int_t^TZ_s\cdot dX_s-\int_t^TdK_s,~ \mathcal P_A-q.s.
\end{equation}

\vspace{0.5em} \noindent
Recall now the notion of solution to this 2BSDE introduced in \cite{soner2012wellposedness} and extended in \cite{possamai2015stochastic}.
\begin{Definition}\label{def:sol2BSDE} We say that a triplet $(Y,Z,K)$ is a solution to the 2BSDE \eqref{eq:2bsde} if there exists $p>1$ such that $(Y,Z,K)\in\mathbb S^p_0(\mathbb F_+^{\mathcal P_A},\mathcal P_A) \times  \mathbb H^p_0(\mathbb F^{\mathcal P_A},\mathcal P_A) \times  \mathbb K^p_0(\mathbb F^{\mathcal P_A},\mathcal P_A)$  satisfies \eqref{eq:2bsde} and $K$ satisfies the minimality condition
\begin{equation}\label{eq:minimality condition}
K_t=\underset{\mathbb P'\in \mathcal P_A[\mathbb P,\mathbb F^+,t]}{ \einf^\mathbb P}\, \mathbb E^{\mathbb P'}\left[K_T \Big| \mathcal F_t^{\mathbb P,+}\right],\; t\in [0,T],\, \mathbb P-a.s.,\, \forall \mathbb P\in \mathcal P_A.
\end{equation}
\end{Definition}

\begin{Remark}
Similarly to \cite{cvitanic2017dynamic}, we use here the result of \cite{nutz2012pathwise} for stochastic integral by considering the aggregative version of the non-decreasing process $K$.
\end{Remark}

\vspace{0.5em} \noindent
From now, we set the standing assumption to be used in all the following results
\begin{Assumption}[\textbf{S}] For some $(\ell,m,\underline m)\in[1,+\infty)\times [\ell,+\infty)\times (0,\ell+m-1]$ with $ \frac{\ell+m}{\underline m+1-\ell}\leq 2$, Assumption $\mathbf{(H)}^{\ell,m,\underline m}$ holds together with Assumptions \ref{assumption:proba}. 
\end{Assumption} 

\vspace{0.5em} \noindent
We have the following result which ensures that the 2BSDE \eqref{eq:2bsde} is well-posed. Its proof is postponed to the Appendix.
\begin{Lemma} \label{lemma:wellposedness 2bsde}
Under Assumption $\mathbf{(S)}$, the 2BSDE \eqref{eq:2bsde} has a unique solution $(Y,Z,K)$ for any $\xi$ in $\mathfrak C$.  
\end{Lemma}

\vspace{0.5em} \noindent
The next Theorem is the main result of this section and it provides an equivalence between solving the Agent's problem \eqref{problem:Agent} and the 2BSDE \eqref{eq:2bsde}. Its proof is postponed to the Appendix and is similar to the proof of \cite[Proposition 5.4]{cvitanic2017dynamic}, being its extension to the worst-case volatility case.

\begin{Theorem}\label{thm:optimaleffort}
Let Assumption $\mathbf{(S)}$ hold and denote by $(Y,Z,K)$ the solution to the 2BSDE \eqref{eq:2bsde}. Then, the value function of the Agent is given by
\begin{equation}\label{resolution:pgAgent}
U_0^A(\xi)= \sup_{\alpha\in \mathcal A} \underset{(\mathbb{P},\nu)\in   \mathcal N_A^\alpha}{{\rm inf}}   \mathbb{E}^\mathbb{P} \left[ Y_0 \right].
\end{equation}
Moreover, $(\alpha^\star,\mathbb P^\star, \nu^\star)\in \mathcal A^\star(\xi)\times \mathcal N_A^{\star}(\xi)$ if and only if $(\alpha^\star,\mathbb P^\star,\nu^\star)\in \mathcal A\times \mathcal N_A$ and satisfies
\begin{itemize}
\item[$(i)$] $(\alpha^\star,\nu^\star)$ attains the sup-inf in the definition of $F^\star(\cdot,X,Y,Z,\widehat{\sigma}^2), \, dt\otimes \mathbb P^\star-$a.e.,
\item[$(ii)$] $K_T=0,\, \mathbb P^\star-$a.s.
\end{itemize}
\end{Theorem} 

\vspace{0.5em}\noindent
To conclude the section, let us comment the intuition behind this result and the limitations of our model. 

\begin{Remark}
If the volatility of the outcome process is fixed and the Agent controls only the drift, it is well-known that his value function is the solution to a BSDE. The worst-case approach of the Agent makes his value function be the infimum of BSDEs and therefore the solution to a 2BSDE. This reasoning works because the Agent controls only the drift and not the volatility of the outcome. Indeed, by considering a controlled volatility coefficient $\sigma(t,x,\alpha,\nu)$, the worst-case approach of the Agent induces a first 2BSDE and the control $\alpha$ induces a second 2BSDE on top of that. Currently, such kind of 2BSDEs has not been studied in the literature.
\end{Remark}

\section{The Principal's Problem}\label{section:principal}

In this section, we aim at solving the contracting problem \eqref{pb:Principal}. This corresponds to an extension of both \cite{cvitanic2017dynamic} to the uncontrolled volatility case and \cite{mastrolia2016moral} in a more general model, without assuming that a dynamic programming principle holds for the value function of the Principal. We follow the ideas of \cite{bayraktar2012stochastic,bayraktar2013stochastic,sirbu2014stochastic}.

\subsection{A pathological stochastic control problem}  \label{section:pathological}

To facilitate the understanding of this section, we provide a general overview of the method we use, dividing it in the following steps.

\vspace{.5em}
\textbf{Step 1.} In Section \ref{section:setadmissiblecontractrewritten}, we rewrite the set of admissible contracts and the Principal's problem \eqref{pb:Principal} making use of the results obtained in Section \ref{section:Agent-problem}. We also make a distinction between the case in which the estimation sets of the Principal and the Agent are disjoint and the case in which they are not.

\vspace{.5em}
\textbf{Step 2.} In Section \ref{section:degenerate}, we show that if the beliefs of the Principal and the Agent are disjoints, there is a degeneracy in the sense that the Principal can propose to the Agent a sequence of admissible contracts such that asymptotically she gets her maximal utility.

\vspace{.5em}
\textbf{Step 3.} We solve next the problem of the Principal in Section \ref{section:pbP} when the beliefs about the volatility of the Principal and the Agent are not disjoint by restricting the study to piece-wise constant controls and by using Perron's method.
\vspace{.5em}

In the following, we suppose that $\mathbf{(S)}$ and the next assumption are enforced. 

\begin{Assumption}[Markovian case]\label{assumption:markov} All the objects considered are Markovian, \textit{i.e.} they depend on $(t,X_\cdot)$ only through $(t,X_t)$.
\end{Assumption}

\begin{Remark}
Assumption \ref{assumption:markov} may be removed if we deal with the theory of path dependent PDEs (see among others \cite{ekren2014viscosity, ren2014}). Here, we assume that it holds for the sake of simplicity and to focus on the procedure to solve the Principal's problem.
\end{Remark}

\subsection{The problem and remark on the set of admissible contracts} \label{section:setadmissiblecontractrewritten}
The solution to the problem of the Agent provides a very particular form for $U_A(\xi)$. More precisely, let $(Y,Z,K)$ be the solution of 2BSDE \eqref{eq:2bsde}, then
\begin{equation} \label{eq:UA-xi}
U_A(\xi)=Y_0-\int_0^T F^\star(s,X_s,Y_s,Z_s,\hat{\sigma}_s^2)ds+\int_0^TZ_s\cdot dX_s+\int_0^TdK_s,~ \mathcal P_A-q.s.,
\end{equation}
the process $K$ satisfies the minimality condition \eqref{eq:minimality condition}, and
$$
\sup_{\alpha \in \mathcal{A}}~\underset{(\P,\nu)\in\mathcal N^\alpha_A }{\inf}   \mathbb{E}^\mathbb{P} \left[ Y_0 \right] \geq R_0.
$$
Let us define the set of $\mathcal{F}_0-$measurable random variables
$$
\mathbb Y_0:= \left\{Y_0,\;   \sup_{\alpha \in \mathcal{A}}~\underset{(\P,\nu)\in\mathcal N^\alpha_A }{\inf}   \mathbb{E}^\mathbb{P} \left[ Y_0 \right] \geq R_0\right\}.
$$
Then, for any contract $\xi\in \Xi$ there exists a triplet $(Y_0,Z,K)\in \mathbb{Y}_0\times  \mathbb H^p_0(\mathbb F^{\mathcal N_A}, \mathcal P_A) \times  \mathbb K^p_0(\mathbb F^{\mathcal N_A}, \mathcal P_A)$ such that \eqref{eq:minimality condition} and \eqref{eq:UA-xi} hold. Since such a triplet is unique, we can establish a one-to-one correspondence between the set of admissible contracts $\Xi$ and an appropriate subset of $\mathbb{Y}_0\times  \mathbb H^p_0(\mathbb F^{\mathcal N_A}, \mathcal P_A) \times  \mathbb K^p_0(\mathbb F^{\mathcal N_A}, \mathcal P_A)$. However, as explained in \cite{mastrolia2016moral}, decomposition \eqref{eq:UA-xi} only holds $\mathcal P_A-$quasi surely and we have to take this fact into account in order to provide a suitable characterization of the set of admissible contracts by means of this formula.

\vspace{.5em}
For any $(Y_0,Z,K)\in \mathbb Y_0\times  \mathbb H^p_0(\mathbb F^{\mathcal N_A}, \mathcal P_A) \times  \mathbb K^p_0(\mathbb F^{\mathcal N_A}, \mathcal P_A)$ such that $K$ satisfies \eqref{eq:minimality condition} and every $(\mathbb P,t)\in \mathcal{P}_A \times [0,T]$, we define the process $Y^{Y_0,Z,K}$  for any $\mathbb P\in \mathcal P_A$ by
\begin{equation}\label{dyn:valuefunction:Agent}Y^{Y_0,Z,K}_t:= Y_0-\int_0^t F^\star(s,X_s,Y^{Y_0,Z,K}_s,Z_s,\hat{\sigma}_s^2)ds+\int_0^tZ_s\cdot dX_s+\int_0^tdK_s, ~ \mathbb{P}-a.s.
\end{equation}

Recall that since $k$ is bounded, $F^\star$ is Lipschitz with respect to $y$, thus $Y^{Y_0,Z,K}$ is well defined. The definition is independent of the probability $\mathbb{P}$ because the stochastic integrals can be defined pathwise (see \cite[Definition 3.2]{cvitanic2017dynamic} and the paragraph which follows). 

\vspace{.5em}
Fix now $Y_0\in\mathbb Y_0$ and let $\mathcal K_{Y_0}$ be the set of pairs $(Z,K) \in \mathbb H^p_0(\mathbb F^{\mathcal N_A}, \mathcal P_A) \times  \mathbb K^p_0(\mathbb F^{\mathcal N_A},\mathcal P_A)$ sufficiently integrable such that $U_A^{-1}(Y_T^{Y_0,Z,K})\in \mathfrak{C}_{\mathcal{P}_A}$ and with $K$ satisfying \eqref{eq:minimality condition}. The Principal has thus to propose a contract with the form $U_A^{-1}(Y^{Y_0,Z,K}_T)$ under every probability measure in the space $\mathcal P_A$. Outside of the support of this space, the Principal is completely free on the salary given to the Agent.

\vspace{0.3em}
\noindent We denote by $\mathcal D$ the set of $\mathcal F_T-$measurable random variables $\xi$ such that 
\begin{equation}\label{decompXi}\xi=\begin{cases}
\displaystyle U_A^{-1}(Y_T^{Y_0,Z,K}),\ \mathcal P_A-q.s.,\\
\displaystyle \widehat\xi,\ \mathcal P_P\backslash\mathcal P_A-q.s.,
\end{cases}
\end{equation}
for some triplet $(Y_0,Z,K)\in \mathbb Y_0\times \mathcal K_{Y_0}$ and some $\widehat{\xi}\in\mathfrak{C}_{\mathcal{P}_P\setminus\mathcal{P}_A}$. The integrability conditions imposed on $Z$, $K$ and $\widehat{\xi}$ ensure us that $\mathcal{D}\subset\Xi$. In fact, from the reasoning given in the paragraphs above we have that $\mathcal{D}$ coincides with $\Xi$ and \eqref{decompXi} corresponds to a characterization of the set of admissible contracts. Therefore, the problem of the Principal \eqref{pb:Principal} becomes

\begin{equation}\label{pb:Principal:general} U_0^P=\sup_{(Y_0,Z,K, \widehat \xi)\in{\mathbb Y_0\times \mathcal K_{Y_0}} \times  \mathfrak C_{\mathcal P_P\setminus \mathcal P_A}} \, U_0^P(U_A^{-1}(Y_T^{Y_0,Z,K}),\widehat{\xi}),
\end{equation}
with the following slight abuse of notations
$$
U_0^P(\mathcal{X},\widehat{\xi}):= \min\left\{   \inf_{(\P,\nu)\in\mathcal N_P^{\alpha^\star(\mathcal{X})}\cap \mathcal N_A^{\alpha^\star(\mathcal X)}} \, \mathbb E^\mathbb P\left[U_P(L(X_T)-\mathcal X) \right], ~
\inf_{(\P,\nu)\in\mathcal N_P^{\alpha^\star(\mathcal X)}\setminus \mathcal N_A^{\alpha^\star(\mathcal X)}} \, \mathbb E^\mathbb P\left[U_P(L(X_T)-\widehat\xi) \right] \right\}.
$$

\subsection{Degeneracies for disjoint believes}\label{section:degenerate}
Similarly to the study made in \cite[Section 4.3.1.]{mastrolia2016moral}, if the believes of the Agent and the Principal are disjoint, we face a pathological case caused by the fact that the Agent and the Principal do not somehow live in the same world. Indeed, if $\mathcal P_A \cap \mathcal P_P=\emptyset$ we have

\begin{equation}\label{pb:Principal:degenerate} U_0^P=\sup_{(Y_0,Z,K, \widehat\xi)\in \mathbb Y_0\times \mathcal K_{Y_0} \times \mathfrak C_{\mathcal P_P}} \, \inf_{(\P,\nu)\in\mathcal N_P^{\alpha^\star (\mathcal X)}} \, \mathbb E^\mathbb P\left[U_P(L(X_T)-\widehat\xi) \right],
\end{equation}
with $\mathcal X=U_A^{-1}(Y_T^{Y_0,Z,K})$. We then have the following proposition.
\begin{Proposition}
If $\mathcal P_P\cap\mathcal P_A=\emptyset$ then $U_0^P=\lim_{x\to\infty} U_P(x)$.
\end{Proposition}
\begin{proof}
Let $M$ be any positive integer and define $\widehat{\xi}^M:=L(X_T)-M$. Take any $(Y_0,Z,K)\in \mathbb{Y}_0 \times \mathcal{K}_{Y_0}$ and set the admissible contract
\begin{equation*}\xi^M:=\begin{cases}
\displaystyle U_A^{-1}(Y_T^{Y_0,Z,K}),\ \mathcal P_A-q.s.,\\
\displaystyle \widehat{\xi}^M,\ \mathcal P_P-q.s.
\end{cases}
\end{equation*}
Then, we have 
$$
U_0^P\geq  \inf_{(\P,\nu)\in \mathcal N_P^{\alpha^\star(\mathcal X)} } \, \mathbb E^\mathbb P\left[U_P(L(X_T)-L(X_T)+M) \right] = U_P(M).
$$
By making $M\to\infty$ we conclude, since the other inequality is trivial. 
\end{proof}

\paragraph{Interpretation.} This result is the same as in \cite[Proposition 4.2]{mastrolia2016moral}. Since the Agent does not see the random variables defined outside of his set of beliefs $\mathcal P_A$, the Principal is completely free on the design of the contract on $\mathcal P_P$. Thus, the Principal can offer a contract which satisfies the reservation utility constraint on $\mathcal P_A$ and which attains asymptotically her maximal utility on $\mathcal P_P$. By doing this the Principal cancels all her risk. This situation is not realistic, since a Principal should not hire an Agent with a completely different point of view on the market behaviour.

\subsection{The Principal's problem with common believes.}\label{section:pbP}

We now turn to a more realistic situation and we study the problem when $\mathcal P_A \cap \mathcal P_P\neq \emptyset$. In this case, as showed in \cite[Proposition 4.3]{mastrolia2016moral},  \eqref{pb:Principal:general} becomes

\begin{equation}\label{pb:Principal:general:intersection} U_0^P=\sup_{Y_0\in \mathbb Y_0} U_0^P(Y_0),
\end{equation}
with the abuse of notation
\begin{equation}\label{pb:Principal:sansY0}
U_0^P(Y_0):= \sup_{(Z,K)\in{\mathcal K_{Y_0}}}  \inf_{(\P,\nu)\in\mathcal N_P^{\alpha^\star(\mathcal X)}\cap \mathcal N_A^{\alpha^\star(\mathcal X)}} \, \mathbb E^\mathbb P\left[U_P\left(L(X_T)-U_A^{-1}(Y_T^{Y_0,Z,K})\right) \right],
\end{equation}
with $\mathcal X=U_A^{-1}(Y_T^{Y_0,Z,K})$.

\subsubsection{A natural restriction to piece-wise constant controls}
 As explained in \cite{sannikov2008continuous}, then in \cite{cvitanic2014moral,cvitanic2017dynamic}, the problem \eqref{pb:Principal:sansY0} coincides with the weak formulation of a (non standard) zero-sum stochastic differential game with the following characteristics
\begin{itemize}
\item control variables: $(Z,K)\in \mathcal K_{Y_0}$ for the Principal and $(\mathbb P,\nu)\in \mathcal N_P^{\alpha^\star(\mathcal X)}\cap \mathcal N_A^{\alpha^\star(\mathcal X)}$ for the Nature,
\item state variables: the output process $X^{x,\Theta}$ and the continuation utility of the Agent $Y^{y,\Theta}$, with dynamic given for any $t\leq s\leq T$, $\mathbb P-a.s.,$ by

\begin{equation}\label{system:sde}\begin{cases}
\displaystyle X^{t,x,\Theta}_s=x+\int_t^s b\left(r,X^{t,x,\Theta}_r,\alpha^\star\left(\mathcal X)\right), \nu_r\right) dr+\int_t^s \sigma(r,X^{t,x,\Theta}_r,\nu_r) dW^{\alpha^\star\left(\mathcal X\right)}_r \vspace{0.3em},\\ 
\displaystyle Y^{t,y,\Theta}_s= y+\int_t^s Z_r\cdot b\left(r,X^{t,x,\Theta}_r,\alpha^\star\left(\mathcal X)\right), \nu_r\right)- F^\star(r,X^{t,x,\Theta}_r,Y^{t,y,\Theta}_r,Z_r,\hat{\sigma}_r^2)dr\vspace{0.3em}\\
\displaystyle \hspace{5.3em}+\int_t^sZ_r\cdot  \sigma(r,X^{t,x,\Theta}_r,\nu_r) dW^{\alpha^\star(\mathcal X)}_r+\int_t^sdK_r,
\end{cases}\end{equation}
with $\Theta\equiv (Z,K,\nu)$ and $\mathcal X=U_A^{-1}(Y_T^{Y_0,Z,K})$.\end{itemize}
We now fix an arbitrary $Y_0\in\mathbb{Y}_0$ and turn to the procedure to solve \eqref{pb:Principal:sansY0}. The main issue is that the class of controls $\mathcal K_{Y_0}$ is too general since, as explained in \cite[Section 4.3.2]{mastrolia2016moral} and \cite{cvitanic2014moral,cvitanic2017dynamic}, the non-decreasing process $K$ impacts the dynamic of $Y^{Y_0,Z,K}$ only throught the minimality condition \eqref{eq:minimality condition} and more information on this process is required to solve the problem. As emphasized in \cite[Remark 3.4]{sirbu2014stochastic}, we need to consider \textit{piecewise controls} and restrict our investigation on elementary strategies. This issue is intrinsically linked to the fact that we are looking for a zero-sum game between the Principal and the Nature. We now consider a restricted set of controls piece-wise constant included in $\mathcal K_{Y_0}$. 
\begin{Definition}[Elementary controls starting at a stopping time]
Let $t\in [0,T]$ and $\tau$ a stopping time $\mathcal G_s^t$-adapted for any $s\in[t, T]$. We say that an $\mathbb R^d\times \mathbb R^+$-valued process $(Z,K)$ (resp.  $\nu\in \mathfrak N$) is an elementary control starting at $\tau$ for the Principal (resp. the Nature) if there exist
\begin{itemize} 
\item a finite sequence $(\tau_i)_{0\leq i\leq n}$ of $\mathcal F_t$-adapted stopping times such that
$$ \tau=\tau_0\leq \dots\leq \tau_n=T,$$
\item a sequence $(z_i,k_i)_{1\leq i\leq n}$ of $\mathbb R^d\times \mathbb R^+$-valued random variables such that $z_i,k_i$ are $\mathcal F_{\tau_{i-1}}^t-$measurable and
$$Z_t=\sum_{i=1}^n z_i \mathbf{1}_{\tau_{i-1}<t\leq \tau_i},\; K_t=\sum_{i=1}^n k_i \mathbf{1}_{\tau_{i-1}<t\leq \tau_i},$$
resp. a sequence $(n_i)_{1\leq i\leq n}$ of $N$-valued random variables such that $n_i$ is $\mathcal F_{\tau_{i-1}}^t-$measurable and
$$\nu_t=\sum_{i=1}^n n_i \mathbf{1}_{\tau_{i-1}<t\leq \tau_i}.$$
\end{itemize}
We denote by $\mathcal U(t,\tau)$ (resp. $\mathcal V(t,\tau)$) the set of elementary controls of the Principal (resp. the Nature). If $\tau=t=0$, we just write $\mathcal U$ (resp. $\mathcal V$). 
\end{Definition}
 
We now set
 $$\mathcal U_{Y_0}:=\mathcal K_{Y_0}\cap \mathcal U, $$
and for any $(Z,K)\in \mathcal U_{Y_0}$
$$\mathcal V_{Y_0,Z,K}:= \left\{ (\P,\nu)\in\mathcal N_P^{\alpha^\star(\mathcal X)}\cap \mathcal N_A^{\alpha^\star(\mathcal X)} \Big| \nu\in \mathcal V\right\}. $$  We thus consider the following restricted problem

\begin{equation}\label{pb:Principal:general:intersection:piecewise} V_0^P=\sup_{Y_0\in \mathbb Y_0} V_0^P(Y_0),
\end{equation}
with the abuse of notation
\begin{equation}\label{pb:Principal:sansY0:piecewise}
V_0^P(Y_0):= \sup_{(Z,K)\in{\mathcal U_{Y_0}}}  \inf_{(\P,\nu)\in \mathcal V_{Y_0,Z,K}} \, \mathbb E^\mathbb P\left[U_P\left(L(X_T)-U_A^{-1}(Y_T^{Y_0,Z,K})\right) \right].
\end{equation}

The literature, an more particularly \cite{sirbu2014stochastic, pham2014two, soner2013dual}, leads us to expect to get $U_0^P=V_0^P$ for particular cases in view of the related papers dealing with this kind of problems. In other words, in some cases, the value of the general problem \eqref{pb:Principal:general:intersection} coincides with its restriction \eqref{pb:Principal:general:intersection:piecewise} to piecewise defined controls. We thus will focus on the restricted problem \eqref{pb:Principal:general:intersection:piecewise} in the following, that we solve completely.

\subsubsection{The intuitive HJBI equation}\label{section:hjbi}

\noindent  
Assumption \textbf{(PPD)} in \cite{mastrolia2016moral} seems to be too complicated to prove\footnote{Another approach not considered in this paper, may consist in proving a \textit{weak} dynamic programming principle by following \cite{bouchard2011weak, bayraktar2013yao}.} for a general class of processes $K$, since it requires a deep study of the measurability of the dynamic version of the value function associated with the problem \eqref{pb:Principal:general:intersection:piecewise}. To avoid this difficulty linked directly to the ambiguity on the volatility of the model, we will deal with the so-called Perron's method by following the same ideas as in \cite{bayraktar2012stochastic,bayraktar2014dynkin,bayraktar2013stochastic,sirbu2014stochastic}.  Recall that if one aims at associating \eqref{pb:Principal:sansY0} with an HJBI equation, as usual in the stochastic control theory, the problem seems to be ill-posed and we need more information on the process $K$. We thus expect to have an optimal contract $\xi:=U_A^{-1}(Y_T^{Y_0,Z,K})$ for which the process $K$ is absolutely continuous. More exactly, and by following \cite[Remark 5.1]{cvitanic2017dynamic} we expect to get an optimal contract in a the subspace of contracts for which there exists a $\mathbb G^{\mathcal N_A}$-predictable process $\Gamma$ with values in $\mathcal M_{d,d}(\mathbb R)$ such that 

\begin{equation}\label{decomp:K}K_t =\int_0^t \left(F^\star(s,Y_s,Z_s,\widehat\sigma^2_s) +\frac12\text{Tr}\left( \widehat\sigma^2_s \Gamma_s\right) - H(s,X_s,Y_s,Z_s,\Gamma_s) \right)ds.
\end{equation}

This intuition leads us to set the following Hamiltonian function $G:[0,T]\times \mathbb R^d\times \mathbb R\times \mathbb R^d\times \mathbb R\times \mathcal S^{d,d}\times \mathbb R\times \mathbb R^d \longrightarrow \mathbb R$ defined by

$$G(t,x,y,p,\tilde p,q,\tilde q,r):=\sup_{(z,\gamma)\in \mathbb R^d \times \mathcal M_{d,d}(\mathbb R)} \inf_{\mathfrak n\in N} g(t,x,y,p,\tilde p,q,\tilde q,r,z,\gamma,\mathfrak n) ,$$
where 
\begin{align*}
g(t,x,y,p,\tilde p,q,\tilde q,r,z,\gamma, \mathfrak n)&:=p\cdot b(t,x,\alpha^\star(t,x,y,z,\widehat\sigma_t),\mathfrak n) +\frac12 \text{Tr}\left(\sigma(t,x,\mathfrak n)\sigma(t,x,\mathfrak n)^\top q \right)\\
&+\tilde p \left(\frac12\text{Tr}\left( \sigma(t,x,\mathfrak n)\sigma(t,x,\mathfrak n)^\top \gamma\right)- H(t,x,y,z,\gamma)\right)\\
&+ \tilde p \, b(t,x,\alpha^\star(t,x,y,z,\widehat\sigma_t),\mathfrak n)\cdot z +   \text{Tr}\left( z^\top \sigma(t,x,\mathfrak n) \sigma(t,x,\mathfrak n)^\top r  \right) \\
&+\frac12 \tilde q ~ \text{Tr}\left( z^\top \sigma(t,x,\mathfrak n)  \sigma(t,x,\mathfrak n)^\top z  \right).
\end{align*}
We can now set the HJBI equation which is hopefully strongly connected to the problem of the Principal \eqref{pb:Principal:sansY0:piecewise}
\begin{equation}\label{HJBI}\begin{cases}
\displaystyle -\partial_t u(t,x,y) - G(t,x,y,\nabla_x u,\partial_y u,\Delta_{xx} u ,\partial_{yy}u ,\nabla_{xy}  u) = 0,\; (t,x,y)\in [0,T)\times \mathbb R^d\times \mathbb R\\
\displaystyle u(T,x,y)=  U_P(L (x)-U_A^{-1}(y)).
\end{cases}
\end{equation}

\subsubsection{Reduction to bounded controls}\label{section:remarkHJBI}
In this section, we study a fundamental property on the Hamiltonian $G$ appearing in the HJBI equation \eqref{HJBI}, in order to simplify the study of the stochastic control problem \eqref{pb:Principal:sansY0:piecewise}. The main difficulty is that the set of controls is unbounded, which can be quite hard to investigate in practice. We thus set an assumption ensuring that the supremum over $(z,\gamma)$ in the definition of $G$ can be reduced to a supremum over a compact set. We show that this assumption holds both for a risk-neutral setting and in the one dimensional case, \textit{i.e.} by assuming that $d=n=1$, with additional growth conditions on the data $b$ and $\sigma$. The reduction of the set of controls is fundamental in the proof of Theorem \ref{main:Principal} below.

\begin{Assumption}\label{continuityR}
For fixed $\tilde p,\tilde q,r\in \mathbb{R}$ and for every $(t,x,y,p,q)\in [0,T]\times \mathbb{R} $ there exists a continuous radius $R:=R(t,x,y,p,q)$ such that
$$
G(t,x,y,p,\tilde p,q,\tilde q,r)= \sup_{|z|\leq R} ~ \sup_{|\gamma|\leq R} ~ \inf_{\mathfrak n\in N} ~ g(t,x,y,p,\tilde p,q,\tilde q,r,z,\gamma,\mathfrak n).
$$

\end{Assumption}

\paragraph{Example 1: Risk-neutral Principal and risk-neutral Agent without discount factor.} 
Assume that both the Principal and the Agent are risk-neutral, \textit{i.e.} $U_P(x)=U_A(x)=x$. We moreover assume that $k\equiv 0$. In this setting, the worst--case measures of both parties coincide, and the problem reduces to a classical drift--control model in the  Principal--Agent literature. In fact, the continuation utility of the Principal at time $t$ is given by
$$u_t^P(x,y):=  \sup_{(Z,K)\in{\mathcal K_{Y_0}}}  \inf_{(\P,\nu)\in\mathcal N_P^{\alpha^\star(\mathcal X)}\cap \mathcal N_A^{\alpha^\star(\mathcal X)}} \, \mathbb E^\mathbb P\left[L(X_T^{t,x,\Theta})-Y_T^{t,y,\Theta} \right], $$
with $x\in \mathbb R^d$ and $y\in \mathbb R$ respectively. Since $k\equiv 0$, it is clear from the system \eqref{system:sde} that $u_t$ is linear with respect to the variable $y$. Roughly speaking, the Hamiltonian $G$ in HJBI equation \eqref{HJBI} is evaluated at $\tilde p=-1$ and $\tilde q=r=0$. Then we have
\begin{align*}
G(t,x,y,p,-1,q,0,0) = & \sup_{(z,\gamma)} \inf_{\mathfrak n\in N} ~ \left( p\cdot b(t,x,\alpha^\star,\mathfrak n) +\frac12 \text{Tr}\left(\sigma\sigma(t,x,\mathfrak n)^\top q \right) \right. \\
&  \hspace{2cm} \left.+ H(t,x,y,z,\gamma) - \frac12\text{Tr}\left( \sigma\sigma(t,x,\mathfrak n)^\top \gamma\right) - b(t,x,\alpha^\star,\mathfrak n)\cdot z \right) \\
\leq & \sup_{(z,\gamma)} \inf_{\mathfrak n\in N}   ~ \left( p\cdot b(t,x,\alpha^\star,\mathfrak n) +\frac12 \text{Tr}\left(\sigma\sigma(t,x,\mathfrak n)^\top q \right) - c(t,x,\alpha^\star) \right) \\
\leq & H(t,x,y,p,q).
\end{align*}

And evaluating the supremum at $z=p$ and $\gamma=q$ we obtain the converse inequality
$$
G(t,x,y,p,-1,q,0,0) \geq H(t,x,y,p,q).
$$
It follows that the optimal controls are $z^\star=p$, $\gamma^\star=q$ and the the infimum is attained for both $H$ and $G$ at the same optimal $\mathfrak n$ denoted by $\nu^\star$. The continuous radius in this case is given by $R(t,x,y,p,q)=\max\{|p|,|q|\}$.

\begin{Remark}
In the risk--neutral problem, an important consequence of the Principal and the Agent having the same worst--case measures is that the the first--best value can be attained. This is a direct consequence, and an extension, of the well known result in the classical drift--control problem.
\end{Remark}

\paragraph{Example 2: One dimensional case.}
Assume that $d=n=1$. The following assumption on the relative growth of the drift $b$, the volatility $\sigma$ and the discount factor $k$ of the output ensures the existence of the continuous radius.
\begin{Assumption} \label{ass:orly}
$b$ and $\sigma$ are continuous functions which satisfy the following properties.
\begin{enumerate}
\item  For every $(t,x,a)\in[0,T]\times\mathbb{R}\times A$ and for every $\underline{\nu}$ global minimum of $\sigma(t,x,\cdot)$, the following limits are finite  
$$
\displaystyle \lim_{\mathfrak n\to \underline{\nu}} ~ \frac{b(t,x,a,\mathfrak n)-b(t,x,a,\underline{\nu})}{\sigma(t,x,\mathfrak n)-\sigma(t,x,\underline{\nu})}, ~ \lim_{\mathfrak n\to \underline{\nu}} ~ \frac{k(t,x,a,\mathfrak n)-k(t,x,a,\underline{\nu})}{\sigma(t,x,\mathfrak n)-\sigma(t,x,\underline{\nu})}.
$$ 
\item For every $(t,x,a)\in[0,T]\times\mathbb{R}\times A$ and for every $\overline{\nu}$ global maximum of $\sigma(t,x,\cdot)$, the following limit is finite 
$$
\displaystyle \lim_{\mathfrak n\to \overline{\nu}} ~ \frac{b(t,x,a,\mathfrak n)-b(t,x,a,\overline{\nu})}{\sigma(t,x,\mathfrak n)-\sigma(t,x,\overline{\nu})}, ~ \lim_{\mathfrak n\to \overline{\nu}} ~ \frac{k(t,x,a,\nu)-k(t,x,a,\overline{\nu})}{\sigma(t,x,\mathfrak n)-\sigma(t,x,\overline{\nu})}.
$$
\end{enumerate}
\end{Assumption}
The proof of the lemma below is postponed to the Appendix \ref{appendix:coercivity}.
\begin{Lemma}\label{lemma:coercivity} \text{ }
 Let Assumption \ref{ass:orly} be satisfied. In addition, assume that $\mathfrak n\mapsto \sigma(t,x,\mathfrak n)$ has a unique minimizer for every $(t,x)$ and $a\mapsto F(t,x,y,z,a,\mathfrak n)$ has a unique maximizer for every $(t,x,y,z,\mathfrak n)$. Then for any $\tilde{q}<0$ and for every $(t,x,y,u,p,\tilde{p},q,r)\in [0,T]\times \mathbb R^7$ Assumption \ref{continuityR} holds.
\end{Lemma}


\subsubsection{Perron's method to solve the Principal problem}\label{section:verif}
We now focus on a deep study of PDE \eqref{HJBI}. We assume that $b$ and $\sigma$ are continuous functions and that Assumption \ref{continuityR} holds. 

\vspace{0.5em} \noindent
In this section we drop the assumptions made in \cite{mastrolia2016moral} and we prove a verification result for a non-smooth value function, by following the Stochastic Perron's method introduced by Bayraktar and S\^irbu  \cite{bayraktar2012stochastic,bayraktar2014dynkin,bayraktar2013stochastic,sirbu2014stochastic}. More precisely, we show that the value function of the Principal associated with the problem \eqref{pb:Principal:sansY0:piecewise} is a viscosity solution to the HJBI equation \eqref{HJBI}. The approach we follow avoids to prove (or assume) a dynamic programming principle and only deals with comparison results. Moreover, the dynamic programming principle is a consequence of the used method. We adapt now the definition of stochastic semi-solutions to stochastic differential games \cite{sirbu2014stochastic} to our framework under the weak formulation.   

\begin{Definition}[Stopping rule]
For $t\in[0,T]$, let $(X,Y)$ be the canonical process on $C\left( \left[ t,T \right], \mathbb{R}^{d+1} \right)$. Define the filtration $\mathbb{B}^t=(\mathcal{B}_s^t)_{t\leq s\leq T}$ by
$$
\mathcal{B}_s^t := 	\sigma( (X(u),Y(u)), t\leq u\leq s ), t\leq s\leq T.
$$
$\tau\in C\left( \left[ t,T \right], \mathbb{R}^{d+1} \right)$ is a stopping rule starting at $t$ if it is a stopping time with respect to $\mathbb{B}^t$.
\end{Definition}

\begin{Definition}[Stochastic semisolutions of the HJBI equation]\label{def:sursous}
Let $Y_0\in \mathbb Y_0$.
\begin{itemize}
\item A function $v:[0,T]\times \mathbb R^d\times \mathbb R\longrightarrow\mathbb R$ is a stochastic sub-solution of HJBI equation \eqref{HJBI} if 
\begin{itemize}
\item[(i-)] $v$ is continuous and $v(T,x,y)\leq U_P( L(x)-U_A^{-1}(y))$ for any $(x,y)\in \mathbb R^d\times \mathbb R,$ 
\item[(ii-)]for any $t\in [0,T]$ and for any stopping rule $\tau\in\mathbb{B}^t$, there exists an elementary control $(\tilde Z,\tilde K)\in \mathcal U_{Y_0}(t,\tau)$ such that for any $(Z,K)\in \mathcal U_{Y_0}(t,t)$, any $(\mathbb P,\nu)\in \mathcal V_{Y_0}(t,t)$ and each stopping rule $\rho\in\mathbb{B}^t$ with $\tau\leq \rho\leq T$ we have
\begin{equation}\label{submartingale}
v(\tau', X_{\tau'}, Y_{\tau'}) \leq \mathbb E^{\mathbb P}\left[v(\rho', X_{\rho'}, Y_{\rho'})\big| \mathcal F_{\tau'}^t \right],\, \mathbb P-a.s.,
\end{equation}
where for any $(x,y,\omega)\in \mathbb R^2\times \Omega$, 
\begin{align*}
X&:= X^{t,x,(Z,K)\otimes_{\tau} (\tilde Z,\tilde K),\nu }, ~ Y:=Y^{t,y,(Z,K)\otimes_{\tau} (\tilde Z,\tilde K),\nu}, \\
\tau'(\omega)&:=\tau(X^{t,x,(Z,K)\otimes_{\tau} (\tilde Z,\tilde K),\nu }_\cdot(\omega), Y^{t,y,(Z,K)\otimes_{\tau} (\tilde Z,\tilde K),\nu}_\cdot(\omega)), \\
\rho'(\omega)&:=\rho(X^{t,x,(Z,K)\otimes_{\rho} (\tilde Z,\tilde K),\nu }_\cdot(\omega), Y^{t,y,(Z,K)\otimes_{\tau} (\tilde Z,\tilde K),\nu}_\cdot(\omega)).
\end{align*}
\end{itemize}
We denote by $\mathbb V^-$ the set of stochastic sub-solution of \eqref{HJBI}.

\item A function $v:[0,T]\times \mathbb R^d\times \mathbb R\longrightarrow\mathbb R$ is a stochastic super-solution of HJBI equation \eqref{HJBI} if 
\begin{itemize}
\item[(i+)] $v$ is continuous and $v(T,x,y)\geq U_P(L(x)-U_A^{-1}(y))$ for any $(x,y)\in \mathbb R^d\times \mathbb R$
\item[(ii+)]for any $t\in [0,T]$, for any stopping rule $\tau\in\mathbb{B}^t$ and for any $(Z,K)\in \mathcal U_{Y_0}(t,t)$, there exists an elementary control $(\hat{\mathbb{P}},\tilde{\nu})\in \mathcal V_{Y_0}(t,\tau)$ such that for every $\nu\in\mathcal{V}(t,t)$ satisfying $(\hat{\mathbb{P}},\nu)\in \mathcal V_{Y_0}(t,t)$ and every stopping rule $\rho\in\mathbb{B}^t$ with $\tau\leq \rho\leq T$ we have
\begin{equation}\label{supermartingale}
v(\tau', X_{\tau'}, Y_{\tau'}) \geq \mathbb E^{\hat{\mathbb{P}} }\left[v(\rho', X_{\rho'}, Y_{\rho'})\big| \mathcal F_{\tau'}^t \right],\, \hat{\mathbb{P}} -a.s.,
\end{equation}
where for any $(x,y,\omega)\in \mathbb R^2\times \Omega$, 
\begin{align*}
X&:= X^{t,x,Z,K,\nu \otimes_{\tau} \tilde{\nu} }, ~ Y:=Y^{t,x,Z,K,\nu \otimes_{\tau} \tilde{\nu} }, \\
\tau'(\omega)&:=\tau(X^{t,x,Z,K,\nu \otimes_{\tau} \tilde{\nu} }_\cdot(\omega), Y^{t,x,Z,K,\nu \otimes_{\tau} \tilde{\nu} }_\cdot(\omega)), \\
\rho'(\omega)&:=\rho(X^{t,x,Z,K,\nu \otimes_{\tau} \tilde{\nu} }_\cdot(\omega), Y^{t,x,Z,K,\nu \otimes_{\tau} \tilde{\nu} }_\cdot(\omega)).
\end{align*}
\end{itemize}
We denote by $\mathbb V^+$ the set of stochastic super-solution of \eqref{HJBI}.
\end{itemize}
\end{Definition} 

To apply Perron's method we need the following assumption, assuring the existence of stochastic semi-solutions to the HJBI equation \eqref{HJBI} (see Assumptions 3.4 and 4.3 in \cite{bayraktar2013stochastic}).

\begin{Assumption}\label{ass:non-empty} 
The sets $\mathbb V^+$ and $\mathbb V^-$ are non-empty.
\end{Assumption}

As explained in \cite{bayraktar2012stochastic,bayraktar2014dynkin} the set $\mathbb V^+$ is trivially non empty if the function $U_P$ is bounded by above, whereas $\mathbb V^-$ is non empty if $U_P$ is bounded by below. \vspace{.5em}

Now we follow the stochastic Perron's method proposed in \cite{sirbu2014stochastic}. Let us define 
$$v^-:= \sup_{v\in \mathbb V^-} v,\; v^+:= \inf_{v\in \mathbb V^+} v,\; $$
and notice that we have from Definition \ref{def:sursous} that for any $Y_0\in \mathbb Y_0$
\begin{equation}\label{ineg:v+-V}v^-(0,x,Y_0)\leq  V_0^P(Y_0) \leq v^+(0,x,Y_0). \end{equation} We thus get the main theorem of this section and we refer to the Appendix \ref{appendix:Principal} for the proof.
\begin{Theorem}\label{main:Principal}
$v^-$ is a lower semi-continuous viscosity super-solution of HJBI equation \eqref{HJBI} and $v^+$ is an upper semi-continuous viscosity sub-solution of HJBI equation \eqref{HJBI}.\vspace{0.3em}

Moreover, if there exists a comparison result for HJBI equation \eqref{HJBI}, \textit{i.e.} for any lower semi-continuous viscosity sub-solution $\underline v$ and for any upper semi-continuous viscosity super-solution $\overline v$, we have 
$$\sup_{[0,T]\times \mathbb R^d\times \mathbb R } (\overline v-\underline v)= \sup_{\mathbb R^d\times \mathbb R } (\overline v(T,\cdot)-\underline v(T,\cdot)),$$ then  
$$v^-(0,x, Y_0)= V_0^P(Y_0)= v^+(0,x, Y_0).$$
\end{Theorem}

\subsubsection{On comparison results in the one-dimensional case for bounded diffusions with quadratic cost.}

In general, it seems to be hard to get a comparison result for HJBI equation \eqref{HJBI} in a very general model and we are convinced that only a case-by-case approach has to be considered. In this section we focus on the one--dimensional case, and we illustrate why we can expect a comparison result for HJBI equation \eqref{HJBI}, when the domain of the equation and the space of controls are bounded. 

\vspace{0.3em} \noindent
For a positive real $\overline a$, let $A:=[0,\overline a]$ be the set of values of control $\alpha$. We assume for technical reasons detailed below (see the footnote) that $U_A$ is a bounded and increasing continuous map. For $M>0$, let $\varphi^M$ be a smooth function taking values in $[0,1]$, such that $\varphi^M(x)=0$ for $|x|\geq M+1$ and $\varphi^M(x)=1$ for $|x|\leq M$. Consider $b(t,x,a,\mathfrak n):=a \varphi^M(x)$ and $\sigma(t,x,\mathfrak n):=\mathfrak n\varphi^M(x)$. In this case, recalling \cite[Remark 3.1]{gong2009viscosity}, the output process $X$ takes values in a bounded set $\mathcal O^X$. We now turn to the process $Y$. Assume that $c(t,x,a)=\frac{a^2}2$. Then, $\alpha^\star(t,x,z)= \pi_A(z\varphi^M(x))$ where $\pi_A$ denotes the projective map from $\mathbb R$ into $A$. Thus 2BSDE \eqref{eq:2bsde} becomes 
$$
Y_t=U_A(\xi)+\int_t^T f(X_s,Z_s)ds-\int_t^TZ_s\cdot dX_s-\int_t^TdK_s,~ \mathcal P_A-q.s,
$$
with $f(x,z):=\left( \pi_A(z \varphi^M(x))\varphi^M(x)z-\frac{ |\pi_A(z \varphi^M(x))|^2}{2}\right)$. Notice that, due to the vertex property together with the definition of $A$, we have
$$
-\frac{|z|^2}{2}\leq f(x,z)\leq \frac{|z|^2}2.
$$
Inspired by \cite[Section 4]{possamai2013second}, we introduce the following purely quadratic 2BSDEs, for $\varsigma\in \{-1,1\}$
\begin{equation}\label{2BSDEentropic}
Y_t^\varsigma=U_A(\xi)+\int_t^T \varsigma\frac{| Z^\varsigma_s|^2}2ds-\int_t^TZ^\varsigma_s\cdot dX_s-\int_t^TdK^\varsigma_s,~ \mathcal P_A-q.s.
\end{equation}
Since $U_A$ is bounded\footnote{The assumption on $U_A$ bounded, instead of $U_A$ concave, is fundamental here to get the comparison theorem. In fact, for quadratic growth BSDE we have to assume the the generator is convex together with exponential moments for the terminal conditions (see \cite{briand2008quadratic}). As far as we now this kind of result does not exist for quadratic 2BSDE, we however think that $U_A$ bounded could be removed by adding finite exponential moments for $U_A(\xi)$ in the definition of $\mathcal C$.}, it follows from \cite[Proposition 4.1]{possamai2013second} that the 2BSDEs admit a unique solution and that there exists a positive constant $\kappa^Y>0$ such that  for all $t\in [0,T]$,  $|Y^\varsigma_t|$ is uniformly bounded by $\kappa^Y$. Hence, we deduce from a comparison principle for 2BSDEs with quadratic growth (see for instance \cite[Proposition 3.1]{possamai2013second}) that $|Y_t|\leq \max_{\varsigma\in \{-1,1\}}(|Y_t^\varsigma|)\leq |\kappa^Y|, \; \forall t\in[0,T], ~ \mathbb P-a.s. $ Thus, the continuation utility of the Agent, being a state variable of the problem of the Principal, is bounded so we can restrict the domain of $y$ in \eqref{HJBI} to $\mathcal O^Y:=[-\kappa^Y,\kappa^Y]$.\vspace{0.5em}

We have shown that we can restrict the domain of the HJBI equation \eqref{HJBI} to a bounded domain $\mathcal O^X\times \mathcal O^Y$. To get now a comparison principle in the sense of Theorem \ref{main:Principal}, we refer to the proof of Lemma 4.3 in S\^irbu \cite{sirbu2014stochastic}. Indeed, by noticing that the required conditions on the parameters are evaluated at the test functions (see Step 3 of the proof of Lemma 4.3 in \cite{sirbu2014stochastic}), the continuity of the radius $R$ in Lemma \ref{lemma:coercivity} ensures that we can reproduce the proof by choosing a penalisation function of the form  $\phi(t,x,y)=e^{-\lambda t}(1+|x|+\psi(|y|))$ where $\psi$ is concave continuously differentiable twice and positive on $\mathcal O^Y$. Therefore, a comparison theorem for HJBI equation \eqref{HJBI} can be obtained and the last part of Theorem \ref{main:Principal} holds.

 \section{Conclusion}
In this work we provide a general comprehensive methodology for Principal-Agent problems with volatility uncertainty and worst-case approach from both sides in a general framework. We characterize the value function of the Agent as the solution to a second--order BSDE. Concerning the problem of the Principal, we rewrite it as a non--standard stochastic differential game that we solve by using Perron's method inspiring by the work of S\^irbu \cite{sirbu2014stochastic}. In this work we extend
\begin{itemize}
\item \cite{mastrolia2016moral} by dropping all the technical assumptions needed and with more general models,
\item \cite{sung2017optimal} by considering general utilities for both the Agent and the Principal in a more general model without any restrictions on the form of the contracts.
\end{itemize}
This work is also a complement of
\begin{itemize}
\item \cite{cvitanic2017dynamic} by considering an uncertainty on the set of probabilities so that the problem of the Principal is more difficult to solve,
\item \cite{sirbu2014stochastic} by adding a Stackelberg equilibrium in the stochastic game.
\end{itemize}

To provide a path for future research, we would like to point that \cite{mastrolia2016moral} has solved a particular \textit{non-learning} model explicitly. Although this assumption allows to get nice closed formula for optimal contracts, it is clearly not realistic at all that the ambiguity set is fixed. In the present paper, we assume that the ambiguity set can evolve along the time but we do not specify how it evolves. An interesting extension to this work might be to add in the problem an adaptive method, inspired by the recent paper \cite{bielecki2017adaptive}, to update the ambiguity set. Indeed, we are convinced that a learning procedure will lead to sharper estimates of the unknown volatility process and that the ambiguity may become negligible after some time.   
 
\section*{Acknowledgments}

The authors thanks a lot Erhan Bayraktar, Dylan Possama\"i and Orlando Rivera for discussions, advices and suggestions concerning this investigation. A part of this work was made when Nicol\'as Hern\'andez-Santib\'a\~nez was a PhD student at Universidad de Chile and Universit\'e Paris-Dauphine. This work benefits from the support of CONICYT, the Chair Financial Risk and the ANR PACMAN.



\begin{thebibliography}{10}

\bibitem{bayraktar2012stochastic}
E.~Bayraktar and M.~S{\^\i}rbu.
\newblock Stochastic {P}erron's method and verification without smoothness
  using viscosity comparison: the linear case.
\newblock {\em Proceedings of the American Mathematical Society},
  140(10):3645--3654, 2012.

\bibitem{bayraktar2013stochastic}
E.~Bayraktar and M.~S{\^\i}rbu.
\newblock Stochastic {P}erron's method for {H}amilton--{J}acobi--{B}ellman
  equations.
\newblock {\em SIAM Journal on Control and Optimization}, 51(6):4274--4294,
  2013.

\bibitem{bayraktar2014dynkin}
E.~{Bayraktar} and M.~{S{\^\i}rbu}.
\newblock {Stochastic Perron's method and verification without smoothness using
  viscosity comparison: obstacle problems and Dynkin games}.
\newblock {\em Proceedings of the American Mathematical Society},
  142(4):1399--1412, 2014.

\bibitem{bayraktar2013yao}
E.~{Bayraktar} and S.~{Yao}.
\newblock {A Weak Dynamic Programming Principle for Zero-Sum Stochastic
  Differential Games with Unbounded Controls}.
\newblock {\em SIAM Journal on Control and Optimization}, 51(3):2036--2080,
  2013.

\bibitem{berge1963topological}
Claude Berge.
\newblock {\em Topological Spaces: including a treatment of multi-valued
  functions, vector spaces, and convexity}.
\newblock Courier Corporation, 1963.

\bibitem{bielecki2017adaptive}
Tomasz~R Bielecki, Tao Chen, Igor Cialenco, Areski Cousin, and Monique
  Jeanblanc.
\newblock Adaptive robust control under model uncertainty.
\newblock {\em arXiv preprint arXiv:1706.02227}, 2017.

\bibitem{bouchard2011weak}
B.~Bouchard and N.~Touzi.
\newblock Weak dynamic programming principle for viscosity solutions.
\newblock {\em SIAM Journal on Control and Optimization}, 49(3):948--962, 2011.

\bibitem{briand2008quadratic}
P.~Briand and Y.~Hu.
\newblock Quadratic {BSDE}s with convex generators and unbounded terminal
  conditions.
\newblock {\em Probability Theory and Related Fields}, 141(3-4):543--567, 2008.

\bibitem{buckdahn2008stochastic}
R.~Buckdahn and J.~Li.
\newblock Stochastic differential games and viscosity solutions of
  {H}amilton--{J}acobi--{B}ellman--{I}saacs equations.
\newblock {\em SIAM Journal on Control and Optimization}, 47(1):444--475, 2008.

\bibitem{buckdahn2014value}
R.~{Buckdahn}, J.~{Li}, and M.~{Quincampoix}.
\newblock {Value in mixed strategies for zero--sum stochastic differential
  games without Isaac's conditions}.
\newblock {\em The Annals of Probability}, 42(4):1724--1768, 2014.

\bibitem{cardaliaguet2013pathwise}
Pierre Cardaliaguet and C~Rainer.
\newblock Pathwise strategies for stochastic differential games with an erratum
  to ``stochastic differential games with asymmetric information''.
\newblock {\em Applied Mathematics \&amp; Optimization}, 68(1):75--84, 2013.

\bibitem{cardaliaguet2009stochastic}
Pierre Cardaliaguet and Catherine Rainer.
\newblock Stochastic differential games with asymmetric information.
\newblock {\em Applied Mathematics \&amp; Optimization}, 59(1):1--36, 2009.

\bibitem{cvitanic2014moral}
J.~Cvitani{\'c}, D.~Possama{\"\i}, and N.~Touzi.
\newblock Moral hazard in dynamic risk management.
\newblock {\em Management Science}, to appear, 2014.

\bibitem{cvitanic2017dynamic}
J.~Cvitani{\'c}, D.~Possama{\"\i}, and N.~Touzi.
\newblock Dynamic programming approach to principal--agent problems.
\newblock {\em Finance and Stochastics,} 22(1):1-37, 2018.

\bibitem{cvitanic2012contract}
J.~Cvitani{\'c} and J.~Zhang.
\newblock {\em Contract theory in continuous--time models}.
\newblock Springer, 2012.

\bibitem{ekren2014viscosity}
I.~Ekren, C.~Keller, N.~Touzi, and J.~Zhang.
\newblock On viscosity solutions of path dependent {PDE}s.
\newblock {\em The Annals of Probability}, 42(1):204--236, 2014.

\bibitem{el1997backward}
N.~El~Karoui, S.~Peng, and M.-C. Quenez.
\newblock Backward stochastic differential equations in finance.
\newblock {\em Mathematical Finance}, 7(1):1--71, 1997.

\bibitem{karoui2013capacities}
N.~El~Karoui and X.~Tan.
\newblock Capacities, measurable selection and dynamic programming part {I}:
  abstract framework.
\newblock {\em arXiv preprint arXiv:1310.3363}, 2013.

\bibitem{karoui2013capacities2}
N.~El~Karoui and X.~Tan.
\newblock Capacities, measurable selection and dynamic programming part {II}:
  application in stochastic control problems.
\newblock {\em arXiv preprint arXiv:1310.3364}, 2013.

\bibitem{emp16}
R.~\'Elie, T.~Mastrolia, and D.~Possama{\"\i}.
\newblock A tale of a principal and many many agents.
\newblock {\em arXiv preprint arXiv:1608.05226}, 2016.

\bibitem{gong2009viscosity}
Ruoting Gong and Christian Houdr{\'e}.
\newblock A viscosity approach to a stochastic control problem on a bounded
  domain.
\newblock {\em arXiv preprint arXiv:0911.0956}, 2009.

\bibitem{hamadene1995backward}
S.~Hamad\`ene and J.-P. Lepeltier.
\newblock Backward equations, stochastic control and zero--sum stochastic
  differential games.
\newblock {\em Stochastics: An International Journal of Probability and
  Stochastic Processes}, 54(3-4):221--231, 1995.

\bibitem{hamadene1997bsdes}
S.~Hamad\`ene, J.-P. Lepeltier, and S.~Peng.
\newblock {BSDE}s with continuous coefficients and stochastic differential
  games.
\newblock In N.~El~Karoui and L.~Mazliak, editors, {\em Backward stochastic
  differential equations}, volume 364 of {\em Chapman \& Hall/CRC Research
  Notes in Mathematics Series}, pages 115--128. Longman, 1997.

\bibitem{holmstrom1987aggregation}
B.~Holmstr{\"o}m and P.~Milgrom.
\newblock Aggregation and linearity in the provision of intertemporal
  incentives.
\newblock {\em Econometrica}, 55(2):303--328, 1987.

\bibitem{isaacs1965differential}
Rufus Isaacs.
\newblock {\em Differential games: a mathematical theory with applications to
  welfare and pursuit, control and optimization}.
\newblock John Wiley and Sons, 1965.

\bibitem{karandikar1995pathwise}
R.L. Karandikar.
\newblock On pathwise stochastic integration.
\newblock {\em Stochastic Processes and their Applications}, 57(1):11--18,
  1995.

\bibitem{lions1988viscosity}
P.-L. Lions.
\newblock Viscosity solutions of fully nonlinear second--order equations and
  optimal stochastic control in infinite dimensions. {P}art i: the case of
  bounded stochastic evolutions.
\newblock {\em Acta Mathematica}, 161(1):243--278, 1988.

\bibitem{lions1989viscosity2}
P.-L. Lions.
\newblock Viscosity solutions of fully nonlinear second--order equations and
  optimal stochastic control in infinite dimensions. iii. {U}niqueness of
  viscosity solutions for general second--order equations.
\newblock {\em Journal of Functional Analysis}, 86(1):1--18, 1989.

\bibitem{lions1989viscosity}
P.-L. Lions.
\newblock Viscosity solutions of fully nonlinear second order equations and
  optimal stochastic control in infinite dimensions. {P}art ii: optimal control
  of {Z}akai's equation.
\newblock In G.~da~Prato and L.~Tubaro, editors, {\em Stochastic partial
  differential equations and applications II. Proceedings of a conference held
  in Trento, Italy February 1--6, 1988}, volume 1390 of {\em Lecture notes in
  mathematics}, pages 147--170. Springer, 1989.

\bibitem{lions1985differential}
P-L Lions and Panagiotis~E Souganidis.
\newblock Differential games, optimal control and directional derivatives of
  viscosity solutions of bellman's and isaacs' equations.
\newblock {\em SIAM journal on control and optimization}, 23(4):566--583, 1985.

\bibitem{mastrolia2016moral}
T.~Mastrolia and D.~Possama{\"\i}.
\newblock Moral hazard under ambiguity.
\newblock {\em To appear in Journal of Optimization Theory and Applications.
  arXiv preprint arXiv:1511.03616}, 2016.

\bibitem{nutz2012pathwise}
M.~Nutz.
\newblock Pathwise construction of stochastic integrals.
\newblock {\em Electronic Communications in Probability}, 17(24):1--7, 2012.

\bibitem{nutz2013constructing}
M.~Nutz and R.~van Handel.
\newblock Constructing sublinear expectations on path space.
\newblock {\em Stochastic Processes and their Applications}, 123(8):3100--3121,
  2013.

\bibitem{pham2014two}
T.~Pham and J.~Zhang.
\newblock Two person zero--sum game in weak formulation and path dependent
  {B}ellman--{I}saacs equation.
\newblock {\em SIAM Journal on Control and Optimization}, 52(4):2090--2121,
  2014.

\bibitem{possamai2015stochastic}
D.~Possama{\"\i}, X.~Tan, and C.~Zhou.
\newblock Stochastic control for a class of nonlinear kernels and applications.
\newblock {\em arXiv preprint arXiv:1510.08439}, 2015.

\bibitem{possamai2013second}
D.~Possama{\"\i} and C.~Zhou.
\newblock Second order backward stochastic differential equations with
  quadratic growth.
\newblock {\em Stochastic Processes and their Applications},
  123(10):3770--3799, 2013.

\bibitem{ren2014}
Z.~Ren, N.~Touzi, and J.~Zhang.
\newblock An overview of viscosity solutions of path--dependent {PDE}s.
\newblock In D.~Crisan, B.~Hambly, and T.~Zariphopoulou, editors, {\em
  Stochastic analysis and applications 2014: in honour of Terry Lyons}, volume
  100 of {\em Springer proceedings in mathematics and statistics}, pages
  397--453. Springer, 2014.

\bibitem{sannikov2008continuous}
Y.~Sannikov.
\newblock A continuous--time version of the principal--agent problem.
\newblock {\em The Review of Economic Studies}, 75(3):957--984, 2008.

\bibitem{sirbu2014stochastic}
M.~S{\^\i}rbu.
\newblock Stochastic {P}erron's method and elementary strategies for zero--sum
  differential games.
\newblock {\em SIAM Journal on Control and Optimization}, 52(3):1693--1711,
  2014.

\bibitem{soner2012wellposedness}
H.M. Soner, N.~Touzi, and J.~Zhang.
\newblock Wellposedness of second order backward {SDE}s.
\newblock {\em Probability Theory and Related Fields}, 153(1-2):149--190, 2012.

\bibitem{soner2013dual}
H.M. Soner, N.~Touzi, and J.~Zhang.
\newblock Dual formulation of second order target problems.
\newblock {\em The Annals of Applied Probability}, 23(1):308--347, 2013.

\bibitem{stroock2007multidimensional}
D.W. Stroock and S.R.S. Varadhan.
\newblock {\em Multidimensional diffusion processes}.
\newblock Springer, 2007.

\bibitem{sung2001lectures}
J.~Sung.
\newblock Lectures on the theory of contracts in corporate finance: from
  discrete--time to continuous--time models.
\newblock {\em Com2Mac Lecture Note Series}, 4, 2001.

\bibitem{sung2017optimal}
J.~Sung.
\newblock Optimal contracting under mean--volatility ambiguity uncertainties:
  an alternative perspective on managerial compensation.
\newblock {\em Available at {SSRN} 2601174}, 2015.

\end{thebibliography}

\appendix

\section{Functional spaces} \label{appendix:spaces}

We finally introduce the spaces used in this paper, by following \cite{possamai2015stochastic}. Let $t\in [0,T]$ and $x\in \Omega$ and a family $(\mathcal P(t,x))_{t\in [0,T]\times x\in \Omega}$ of sets of probability measures on $(\Omega,\mathcal F_T)$. In this section, we denote by $\mathbb X:= (\mathcal X_s)_{s\in [0,T]}$ a general filtration on $(\Omega,\mathcal F_T)$. For any $\mathcal X_T-$measurable real valued random variable $\xi$ such that $\sup_{\mathbb P\in \mathcal P(t,x)} \mathbb E^\mathbb P[|\xi|]<+\infty$, we set for any $s\in [t,T]$
$$\mathbb E_s^{\mathbb P,t,x,\mathbb X^+}[\xi]:= \underset{\mathbb P'\in \mathcal P (t,x)[\mathbb P,\mathbb X^+,t]}{\esup^\mathbb P} \mathbb E^{\mathbb P'}[\xi | \mathcal X_s].$$

Let $p\geq 1$ and $\mathbb P\in \mathcal P(t,x)$ and $\mathbb X_\mathbb P$ the usual $\mathbb P$-augmented filtration associated with $\mathbb X$.
\begin{itemize}
\item Let $\kappa\in [1,p]$, $\mathbb L^{p,\kappa}_{t,x}(\mathbb X,\mathcal P)$ denotes the space of $\mathcal X_T-$measurable $\mathbb R-$valued random variables $\xi$ such that
$$\| \xi\|^p_{\mathbb L^{p,\kappa}_{t,x}(\mathbb X,\mathcal P)}:= \sup_{\mathbb P\in \mathcal P(t,x)} \mathbb E^\mathbb P\left[ \underset{t\leq s\leq T}{\esup^\mathbb P} \left( \mathbb E_s^{\mathbb P,t,x,\mathbb F^+}[|\xi|^\kappa]\right)^{\frac p\kappa}\right]<+\infty.$$
\item $\mathbb H^p_{t,x}(\mathbb X,\mathbb P)$ denotes the spaces of $\mathbb X$-predictable $\mathbb R^d$-valued processes $Z$ such that
$$\| Z\|_{\mathbb H^p_{t,x}(\mathbb X, \mathbb P)}^p:=\mathbb E^\mathbb P\left[\left( \int_t^T \| \widehat{\sigma}_s^\frac12Z_s\|^2 ds\right)^\frac p2 \right] <+\infty. $$
We denote by $\mathbb H^p_{t,x}(\mathbb X,\mathcal P)$ the spaces of $\mathbb X$-predictable $\mathbb R^d$-valued processes $Z$ such that
$$\| Z\|_{\mathbb H^p_{t,x} (\mathbb X,\mathcal P)   }^p:=\sup_{\mathbb P\in \mathcal P(t,x)} \| Z\|_{\mathbb H^p_{t,x}(\mathbb P)}^p <+\infty. $$
\item $\mathbb S^p_{t,x}(\mathbb X,\mathbb P)$ denotes the spaces of $\mathbb X$-progressively measurable $\mathbb R$-valued processes $Y$ such that
$$\| Y\|_{\mathbb S^p_{t,x}(\mathbb X, \mathbb P)}^p:=\mathbb E^\mathbb P\left[\sup_{s\in [t,T]} |Y_s|^p \right] <+\infty. $$
We denote by $\mathbb S^p_{t,x}(\mathbb X,\mathcal P)$ of $\mathbb X$-progressively measurable $\mathbb R$-valued processes  $Y$ such that
$$\| Y\|_{\mathbb S^p_{t,x} (\mathbb X,\mathcal P)}^p:=\sup_{\mathbb P\in \mathcal P(t,x)} \| Y\|_{\mathbb S^p_{t,x}(\mathbb P)}^p <+\infty. $$
\item $\mathbb K^p_{t,x}(\mathbb X,\mathbb P)$ denotes the spaces of $\mathbb X$-optional $\mathbb R$-valued processes $K$ with $\mathbb P-$a.s. c\`adl\`ag and non-decreasing paths on $[t,T]$ with $K_t=0,\, \mathbb P-a.s.$ and
$$\| K\|_{\mathbb K^p_{t,x}(\mathbb X, \mathbb P)}^p:=\mathbb E^\mathbb P\left[K_T^p \right] <+\infty. $$
We denote by $\mathbb K^p_{t,x}((\mathbb X_\mathbb P)_{\mathbb P \in \mathcal{P}(t,x)})$ the set of families of processes $(K^\mathbb P)_{\mathbb P\in \mathcal P(t,x)}$ such that for any $\mathbb P\in \mathcal P(t,x)$, $K^\mathbb P\in \mathbb K^p_{t,x}(\mathbb X_\mathbb P,\mathbb P)$ and
$$\sup_{\mathbb P\in \mathcal P(t,x)} \|K^\mathbb P\|_{\mathbb K^p_{t,x}(\mathbb P)}<+\infty. $$

\item $\mathbb M^p_{t,x}(\mathbb X,\mathbb P)$ denotes the spaces of $\mathbb X$-optional $\mathbb R$-valued martingales $M$ with $\mathbb P-$a.s. c\`adl\`ag paths on $[t,T]$ with $M_t=0,\, \mathbb P-a.s.$ and
$$\| M\|_{\mathbb M^p_{t,x}(\mathbb X, \mathbb P)}^p:=\mathbb E^\mathbb P\left[[M]_T^{\frac{p}2} \right] <+\infty. $$
We denote by $\mathbb M^p_{t,x}((\mathbb X_\mathbb P)_{\mathbb P \in \mathcal{P}(t,x)})$ the set of families of processes $(M^\mathbb P)_{\mathbb P\in \mathcal P(t,x)}$ such that for any $\mathbb P\in \mathcal P(t,x)$, $M^\mathbb P\in \mathbb M^p_{t,x}(\mathbb X_\mathbb P,\mathbb P)$ and
$$\sup_{\mathbb P\in \mathcal P(t,x)} \|M^\mathbb P\|_{\mathbb M^p_{t,x}(\mathbb P)}<+\infty. $$
\end{itemize}

When $t=0$ we simplify the previous notations by omitting the dependence on $x$.

\section{Proofs for the Agent's problem}

\begin{proof}[Proof of Lemma \ref{lemma:wellposedness 2bsde}.] Since $\frac{\ell+m}{\underline m+1-\ell}\leq 2$, we have from Lemma \ref{lemmagrowth} that the 2BSDE \eqref{eq:2bsde} has quadratic growth with respect to $z$ and coincide with the framework of \cite{possamai2013second}. In view of Remark 4.2 in \cite{possamai2015stochastic}, we aim at applying Theorem 4.1 in \cite{possamai2015stochastic} by slightly changing its assumptions. More precisely, we replace (i) of Assumption 2.1 in \cite{possamai2015stochastic} by Assumption 2.1 in \cite{possamai2013second}, excepting part (iii). Condition (iv) in \cite{possamai2013second} is a consequence of Lemma \ref{lemmagrowth}. Conditions (v)-(vi) in \cite{possamai2013second} holds in our setting because $k$ is bounded. Therefore, Assumption 2.1 in \cite{possamai2013second} is satisfied.

\vspace{0.3em} \noindent 
Finally, we turn to parts (ii)--(v) of Assumption 2.1 in \cite{possamai2015stochastic}. The terminal condition $U_A(\xi)$ belongs to $\mathbb L^{p,\kappa}_{0,x}(\mathcal{P}_A)$ by definition of the admissible contracts and the conditions imposed on $c$ in $(\mathbf H^{\ell,m,\underline m})$ ensure that (ii) holds. The parts (iii), (iv) and (v) correspond exactly to our Assumption \ref{assumption:proba}. \end{proof}


\vspace{0.5em}
\begin{proof}[Proof of Theorem \ref{thm:optimaleffort}.]  We first prove that \eqref{resolution:pgAgent} holds with a characterization of the optimal effort of the Agent as a maximizer of the 2BSDE \eqref{eq:2bsde}. The proof is divided in 4 steps. 

\vspace{0.5em}
\hspace{2em}{$\bullet$ \bf Step 1:} For every $(\alpha,\nu)\in\mathcal{A} \times\mathcal{V}(\hat{\sigma}^2)$ denote by $(Y^{\alpha,\nu},Z^{\alpha,\nu},K^{\alpha,\nu})$ the solution of the following controlled 2BSDE, defined $\mathcal P_A-q.s.$ (well-posedness holds by the same arguments employed in the proof of Lemma \ref{lemma:wellposedness 2bsde}) 
\begin{equation}\label{eq:2bsde controlled} Y_t^{\alpha,\nu}=U_A(\xi)+\int_t^TF(s,X,Y^{\alpha,\nu}_s,Z^{\alpha,\nu}_s,\alpha_s,\nu_s)ds-\int_t^TZ_s^{\alpha,\nu}\cdot dX_s-\int_t^TdK^{\alpha,\nu}_s- \int_0^TdM_s^{\alpha,\nu}.
\end{equation}
Consider also, for every $\alpha\in\mathcal{A}$, the solution $(Y^{\alpha},Z^{\alpha},K^{\alpha})$ of the following 2BSDE, defined $\mathcal P_A-q.s.$
\begin{equation}\label{eq:2bsde effort} Y_t^{\alpha}=U_A(\xi)+\int_t^T \underset{\nu \in V_s(x,\hat{\sigma}_s^2) }{ {\rm inf} } F(s,X,Y^{\alpha}_s,Z^{\alpha}_s,\alpha_s,\nu)ds-\int_t^TZ_s^{\alpha}\cdot dX_s-\int_t^TdK^{\alpha}_s- \int_0^TdM_s^{\alpha}.
\end{equation}
We have from comparison theorems for 2BSDEs (which are inherited by the classical comparison results for BSDEs) 
\begin{align} \nonumber
Y_0 & = \underset{\alpha \in\mathcal{A}}{{\rm ess~sup}^\mathbb{P}} ~ Y_0^{\alpha},~\mathbb{P}-\textrm{a.s. for every } \mathbb{P}\in \mathcal P_A \\
& = \underset{\alpha \in\mathcal{A}}{{\rm ess~sup}^\mathbb{P}} ~ \underset{\nu\in\mathcal{V}(\hat{\sigma}^2)}{{\rm ess~inf}^\mathbb{P}} ~ Y_0^{\alpha,\nu},~\mathbb{P}-\textrm{a.s. for every } \mathbb{P}\in \mathcal P_A. \label{eq:2BSDE-maximizers-Fstar}
\end{align}

\vspace{0.5em}
\hspace{2em}{$\bullet$ \bf Step 2:} Next, consider for any $\mathbb{P}\in \mathcal P_A$ the triple $(\mathcal{Y}_t^{\mathbb{P},\alpha,\nu},\mathcal{Z}_t^{\mathbb{P},\alpha,\nu},\mathcal{M}_t^{\mathbb{P},\alpha,\nu})_{t\in[0,T]}$ which is the solution of the (well-posed) linear BSDE
\begin{equation}\label{eq:bsde controlled}
\mathcal{Y}_0^{\mathbb{P},\alpha,\nu}=U_A(\xi)+\int_0^TF(s,X,\mathcal{Y}_s^{\mathbb{P},\alpha,\nu},\mathcal{Z}_s^{\mathbb{P},\alpha,\nu},\alpha_s,\nu_s)ds-\int_0^T\mathcal{Z}_s^{\mathbb{P},\alpha,\nu}\cdot dX_s-\int_0^Td\mathcal{M}_s^{\mathbb{P},\alpha,\nu},~ \mathbb{P}-a.s.
\end{equation}

\vspace{0.5em} \noindent
We will follow the idea of Theorem 4.2 in \cite{possamai2015stochastic}, to prove that for every $(\alpha,\nu)\in \mathcal{A}\times\mathcal{V}(\hat{\sigma}^2)$, the solution of the 2BSDE \eqref{eq:2bsde controlled} satisfies the following representation
\begin{equation}\label{eq:2BSDE representation}
Y_0^{\alpha,\nu} = \underset{\mathbb{P'}\in \mathcal P_A[\mathbb P,\mathbb F^+,0]}{{\rm ess~inf}^\mathbb{P}} ~\mathcal Y_0^{\mathbb{P'},\alpha,\nu},~\mathbb{P}-a.s. \textrm{ for every } \mathbb{P}\in \mathcal P_A.
\end{equation}

\vspace{0.5em} \noindent
First, notice that since $K^{\alpha,\nu}$ is non-decreasing, we have for every $ \mathbb{P}\in \mathcal P_A$ and $\mathbb{P'}\in \mathcal P_A[\mathbb P,\mathbb F^+,0]$
$$
Y_0^{\alpha,\nu}\leq ~\mathcal Y_0^{\mathbb{P'},\alpha,\nu},\; \mathbb P-a.s.
$$
thus 
$$
Y_0^{\alpha,\nu} \leq \underset{\mathbb{P'}\in \mathcal P_A[\mathbb P,\mathbb F^+,0]}{{\rm ess~inf}^\mathbb{P}} ~\mathcal Y_0^{\mathbb{P'},\alpha,\nu},~\mathbb{P}-a.s.\textrm{ for every } \mathbb{P}\in \mathcal P_A.
$$
To the reverse inequality, compute for every $\mathbb{P}\in \mathcal P_A$
\begin{align*}
\mathcal{Y}_t^{\mathbb{P},\alpha,\nu}-Y_t^{\alpha,\nu} &= \int_t^T\left( F(s,X,\mathcal{Y}_s^{\mathbb{P},\alpha,\nu},\mathcal{Z}_s^{\mathbb{P},\alpha,\nu},\alpha_s,\nu_s)  -  F(s,X,Y^{\alpha,\nu}_s,Z^{\alpha,\nu}_s,\alpha_s,\nu_s) \right)ds\\
&-\int_t^T\left( \mathcal{Z}_s^{\mathbb{P},\alpha,\nu} - Z_s^{\alpha,\nu} \right)\cdot dX_s+\int_t^TdK^{\alpha,\nu}_s - \int_t^T \left( d\mathcal M^{\mathbb P,\alpha,\nu}_s - dM^{\alpha,\nu}_s \right), \mathbb{P}-a.s.
\end{align*}
Which is equivalent to
\begin{align*}
\mathcal{Y}_t^{\mathbb{P},\alpha,\nu}-Y_t^{\alpha,\nu} &= \int_t^T\left(-k(s,X,\alpha_s,\nu_s)(\mathcal{Y}_s^{\mathbb{P},\alpha,\nu}-Y^{\alpha,\nu}_s) + b(s,X,\alpha_s) \left( \mathcal{Z}_s^{\mathbb{P},\alpha,\nu} - Z_s^{\alpha,\nu} \right) \right)ds\\
&-\int_t^T\left( \mathcal{Z}_s^{\mathbb{P},\alpha,\nu} - Z_s^{\alpha,\nu} \right)\cdot dX_s+\int_t^TdK^{\alpha,\nu}_s - \int_t^T \left( d\mathcal M^{\mathbb P,\alpha,\nu}_s - dM^{\alpha,\nu}_s \right), \mathbb{P}-a.s.
\end{align*}
Using a linearization (see for instance \cite{el1997backward}), we get
$$
\mathcal{Y}_0^{\mathbb{P},\alpha,\nu}-Y_0^{\alpha,\nu}=\E^{\P^{\alpha}}\left[ \int_0^T \mathcal K_{0,s}^{\alpha,\nu}  dK_s^{\alpha,\nu}~\Bigg|~ \mathcal{F}_0  \right],~ \mathbb{P}-a.s.
$$
Then, from Assumption $(\mathbf H^{\ell,m,\underline m})$ $(iii)$ we deduce that 
$$
\mathcal{Y}_0^{\mathbb{P},\alpha,\nu}-Y_0^{\alpha,\nu}\geq e^{-\kappa T}\E^{\P^{\alpha}}\left[ \int_0^T dK^{\alpha,\nu}_s~\Bigg|~ \mathcal{F}_0  \right],~ \mathbb{P}-a.s.
$$
Since $K^{\alpha,\nu}$ satisfies the minimality condition \eqref{eq:minimality condition}, we deduce that $\mathcal{Y}_0^{\mathbb{P},\alpha,\nu}-Y_0^{\alpha,\nu}\geq0,\; \mathbb{P}-a.s.$ for every $\mathbb{P}\in \mathcal P_A$ and \eqref{eq:2BSDE representation} holds.

\vspace{0.5em}
\hspace{2em}{$\bullet$ \bf Step 3:}
Finally, by denoting $c_s^\alpha:= c(s,X,\alpha_s)$, $k_s^{\alpha,\nu}=k(s,X,\alpha_s,\nu_s)$, $b_s^{\alpha,\nu}=b(s,X,\alpha_s,\nu_s)$, we can rewrite the BSDE \eqref{eq:bsde controlled} $\mathbb{P}-a.s.$ as
\begin{equation*}
\mathcal{Y}_0^{\mathbb{P},\alpha,\nu} = U_A(\xi)+\int_0^T \left( - k_s^{\alpha,\nu} \mathcal{Y}_s^{\mathbb{P},\alpha,\nu} - c_s^\alpha + \mathcal{Z}_s^{\mathbb{P},\alpha,\nu} \cdot b_s^{\alpha,\nu}  \right) ds-\int_0^T\mathcal{Z}_s^{\mathbb{P},\alpha,\nu}\cdot dX_s-\int_0^Td\mathcal{M}_s^{\mathbb{P},\alpha,\nu}.
\end{equation*}
Which is equivalent to
\begin{align*}
\mathcal{Y}_0^{\mathbb{P},\alpha,\nu}  = \ & U_A(\xi)+ \int_0^T \left( -k_s^{\alpha,\nu} \mathcal{Y}_s^{\mathbb{P},\alpha,\nu} - c_s^\alpha  + (\sigma_s^\nu)^\top \mathcal{Z}_s^{\mathbb{P},\alpha,\nu} \cdot (\sigma_s^\nu)^\top \left( \sigma_s^\nu \sigma_s^{\nu^\top} \right)^{-1} b_s^{\alpha,\nu}  \right) ds \\
 & -  \int_0^T (\sigma_s^\nu)^\top \mathcal{Z}_s^{\mathbb{P},\alpha,\nu}\cdot  dW_s^{\mathbb P}-\int_0^Td\mathcal{M}_s^{\mathbb{P},\alpha,\nu}.
\end{align*}
Defining $\mathcal{\hat{Z}}_s^{\mathbb{P},\alpha,\nu} = (\sigma_s^\nu)^\top \mathcal{Z}_s^{\mathbb{P},\alpha,\nu}$, we obtain
\begin{align*}
\mathcal{Y}_0^{\mathbb{P},\alpha,\nu}  = \ & U_A(\xi) + \int_0^T \left( - k_s^{\alpha,\nu} \mathcal{Y}_s^{\mathbb{P},\alpha,\nu} - c_s^\alpha + \mathcal{\hat{Z}}_s^{\mathbb{P},\alpha,\nu} \cdot (\sigma_s^\nu)^\top \left( \sigma_s^\nu \sigma_s^{\nu^\top} \right)^{-1} b_s^{\alpha,\nu}  \right) ds \\
 & -  \int_0^T \mathcal{\hat{Z}}_s^{\mathbb{P},\alpha,\nu} \cdot  dW_s^{\mathbb P}-\int_0^Td\mathcal{M}_s^{\mathbb{P},\alpha,\nu},
\end{align*}
whose solution is 
$$
\mathcal Y_0^{\mathbb P,\alpha,\nu}=\E^{\mathbb{P}^{\alpha,\nu}}\left[ \mathcal K_{0,T}^{\alpha,\nu} U_A(\xi)-\int_0^T \mathcal K_{0,s}^{\alpha,\nu}  c_s^\alpha ds~\Bigg|~ \mathcal{F}_0  \right],~ \mathbb{P}-a.s.,
$$
where the measure $\mathbb{P}^{\alpha,\nu}$ is equivalent to $\mathbb{P}$ and is defined by
$$
\frac{d\mathbb{P}^{\alpha,\nu}}{d\mathbb P}:= \mathcal E\left(\int_0^T   \sigma^\top(\sigma\sigma^\top)^{-1}(s,X,\nu_s) b(s,X,\alpha_s,\nu_s) \cdot dW_s^{\mathbb P}\right).
$$

\vspace{0.5em}
\hspace{2em}{$\bullet$ \bf Step 4:} We have from the previous steps that the measure $\mathbb{P}^{\alpha,\nu}\in \mathcal{P}_A^\alpha$ and for every measure $\mathbb{P}\in \mathcal{P}_A$ we have $\mathbb{P}$-a.s. 
\begin{align*}
Y_0 & = \underset{\alpha \in\mathcal{A}}{{\rm ess~sup}^\mathbb{P}} ~ \underset{\nu\in\mathcal{V}(\hat{\sigma}^2)}{{\rm ess~inf}^\mathbb{P}} ~ \underset{\mathbb{P'}\in \mathcal P_A[\mathbb P,\mathbb F^+,0]}{{\rm ess~inf}^\mathbb{P}} \E^{\mathbb{P'^{\alpha,\nu}}}\left[ \mathcal K_{0,T}^{\alpha,\nu} U_A(\xi) - \int_0^T \mathcal K_{0,s}^{\alpha,\nu}  c_s^\alpha ds~\Bigg|~ \mathcal{F}_0 \right] \\
& =  \underset{\alpha\in\mathcal{A}}{{\rm ess~sup}^\mathbb{P}}~ \underset{(\mathbb{P'},\nu)\in \mathcal N_A^\alpha[\mathbb P,\mathbb F^+,0]}{{\rm ess~inf}^\mathbb{P}} \E^{\mathbb{P'}}\left[ \mathcal K_{0,T}^{\alpha,\nu} U_A(\xi) - \int_0^T \mathcal K_{0,s}^{\alpha,\nu}  c_s^\alpha ds~\Bigg|~ \mathcal{F}_0 \right] .
\end{align*}
By similar arguments to the ones used in the proofs of Lemma 3.5 and Theorem 5.2 of \cite{possamai2015stochastic}, it follows that
\begin{align*}
\sup_{\alpha \in \mathcal{A}}~\underset{(\P,\nu)\in\mathcal N^\alpha_A }{\inf}   \mathbb{E}^\mathbb{P} \left[ Y_0 \right] &  = 	\sup_{\alpha \in \mathcal{A}}~\underset{(\mathbb \P,\nu)\in\mathcal N^\alpha_A }{\inf}\E^\P\left[ \mathcal K_{0,T} U_A(\xi) - \int_0^T \mathcal K_{0,s}  c_s^\alpha ds \right] \\
& = U_0^A(\xi).
\end{align*}

\vspace{0.5em} \noindent
We now turn to the second part of the Theorem with the characterization of an optimal triplet $(\alpha,\mathbb P,\nu)$ for the optimization problem \eqref{resolution:pgAgent}.
From the proof of the first part, it is clear that a control $(\alpha^\star,\mathbb{P}^\star,\nu^\star)$ is optimal if and only if it attains all the essential suprema and infima above. The infimum in \eqref{eq:2BSDE representation} is attained if $(ii)$ holds and equality \eqref{eq:2BSDE-maximizers-Fstar} holds if $\alpha^\star$ and $\nu^\star$ satisfy $(i)$.
\end{proof}


\section{Proof of Lemma \ref{lemma:coercivity}}\label{appendix:coercivity}

The proof of Lemma \ref{lemma:coercivity} is based on the following Lemma.
\begin{Lemma} \label{lemma:orly}
Let $\sigma:[c,d]\rightarrow \mathbb{R}$ be continuous, strictly positive and let $q:[c,d]\rightarrow \mathbb{R}$ be continuous. Define for every $\gamma\in\mathbb{R}$ the map $f_\gamma(x) := \gamma \sigma(x)^2 - q(x)$. 
\begin{enumerate}
\item Let $\underline{x}$ be a minimizer of $\sigma$, which maximizes $q$ between all the minimizers of $\sigma$. If the following limit is finite  
$$
\ell:= 	\displaystyle \lim_{x\to \underline{x}} ~ \frac{q(x)-q(\underline{x})}{\sigma(x)-\sigma(\underline{x})},
$$ 
then there exists $M>0$ such that $f_\gamma$ attains its minimum over $[c,d]$ at $\underline{x}$ for every $\gamma>M$.
\item Let $\overline{x}$ be a maximizer of $\sigma$, which minimizes $q$ between all the maximizers of $\sigma$. If the following limit is finite  
$$
L:= 	\displaystyle \lim_{x\to \overline{x}} ~ \frac{q(x)-q(\overline{x})}{\sigma(x)-\sigma(\overline{x})},
$$
then there exists $m<0$ such that $f_\gamma$ attains its minimum over $[c,d]$ at $\overline{x}$ for every $\gamma<m$.
\end{enumerate}
\end{Lemma}

\begin{proof}[\bf Proof]
\begin{enumerate}
\item We suppose without loss of generality that $\sigma$ attains its minimum over $[c,d]$ at a unique point $\underline{x}$. Define $g:[c,d]\rightarrow\mathbb{R}$ by
$$
g(x) = \left\{ \begin{array}{cc} \frac{q(x)-q(\underline{x})}{\sigma(x)-\sigma(\underline{x})}, & x\neq \underline{x}, \\
\ell , & x=\underline{x}.
\end{array}  \right.
$$
We have that $g$ is continuous on $[c,d]$ and therefore there exists $M_g$ such that 
$$ |g(x)|\leq M_g,~ \forall x\in[c,d].$$
Then, for every $\gamma>M:= \frac{M_g}{2 \sigma(\underline{x})}$ we have
\begin{align*}
& \gamma > \frac{q(x)-q(\underline{x})}{2\sigma(\underline{x})(\sigma(x)-\sigma(\underline{x}))}, & \forall x\in[c,d], x\neq \underline{x}, \\
\iff & \gamma \cdot 2\sigma(\underline{x})(\sigma(x)-\sigma(\underline{x})) > q(x)-q(\underline{x}),& \forall x\in[c,d], x\neq \underline{x}, \\
\Longrightarrow &  \gamma(\sigma(x)+\sigma(\underline{x}))(\sigma(x)-\sigma(\underline{x}))> q(x)-q(\underline{x}),& \forall x\in[c,d], x\neq \underline{x}, \\
\iff & \gamma \sigma(x)^2 - q(x) > \gamma \sigma(\underline{x})^2 - q(\underline{x}), & \forall x\in[c,d], x\neq \underline{x}.
\end{align*}

\item We suppose without loss of generality that $\sigma$ attains its maximum over $[c,d]$ at a unique point $\overline{x}$. Define $G:[c,d]\rightarrow\mathbb{R}$ by
$$
G(x) = \left\{ \begin{array}{cc} \frac{q(x)-q(\overline{x})}{\sigma(x)-\sigma(\overline{x})}, & x\neq \overline{x}, \\
L , & x=\overline{x}.
\end{array}  \right.
$$
We have that $G$ is continuous on $[c,d]$ and therefore there exists $M_G$ such that 
$$ |G(x)|\leq M_G,~ \forall x\in[c,d].$$
Then, for every $\gamma<m:= -\frac{M_G}{2 \sigma(\underline{x})}$ we have
\begin{align*}
& \gamma < \frac{q(\overline{x})-q(x)}{2\sigma(\underline{x})(\sigma(\overline{x})-\sigma(x))}, & \forall x\in[c,d], x\neq \underline{x}, \\
\iff & \gamma \cdot 2\sigma(\underline{x})(\sigma(\overline{x})-\sigma(x)) < q(\overline{x})-q(x),& \forall x\in[c,d], x\neq \underline{x}, \\
\Longrightarrow &  \gamma(\sigma(\overline{x})+\sigma(x))(\sigma(\overline{x})-\sigma(x)) < q(\overline{x})-q(x),& \forall x\in[c,d], x\neq \overline{x}, \\
\iff &  \gamma \sigma(\overline{x})^2 - q(\overline{x}) < \gamma \sigma(x)^2 - q(x), & \forall x\in[c,d], x\neq \overline{x}.
\end{align*}
\end{enumerate}

\end{proof}

\begin{proof}[\bf Proof of Lemma \ref{lemma:coercivity}.] ~ \\ 

\noindent (a) If $\tilde{q}<0$, the boundedness of $b$ and $\sigma$, together with Lemma \ref{lemmagrowth} makes $g$ coercive in $z$ and the supremum in this variable can be restricted to a compact. The property on $\gamma$ is independent of the sign of $\tilde{q}$ and is presented next. Recall the Hamiltonian
$$
H(t,x,y,z,\gamma) = \sup_{a\in  A} \inf_{\mathfrak n \in N}\left\{\frac12 \gamma \sigma(t,x,\mathfrak n )^2 - k(t,x,a,\mathfrak n )y - c(t,x,a) + b(t,x,a,\mathfrak n )z \right\}.
$$
It follows from Lemma \ref{lemma:orly} the existence of $m,M\in\mathbb{R}$ such that if $\gamma>M$ then the infimum in $H$ is attained at the minimizer $\underline{\nu}$ of $\sigma(t,x,\cdot)$ and if $\gamma<m$ then the infimum in $H$ is attained at the maximizer $\overline{\nu}$ of $\sigma(t,x,\cdot)$.

\vspace{0.5em}\noindent
Suppose now that $\tilde{p}>0$. It follows again from Lemma \ref{lemma:orly}, that for $\gamma$ big enough, the infimum in $G$ is attained at the minimizer $\underline{\nu}$. For $\gamma$ negative enough, the infimum in $G$ is attained at the maximizer $\overline{\nu}$. This means that there exists some $R:=R(t,x,y,u,p,\tilde{p},q,\tilde{q},r)$ such that $G$ and $H$ attain its minima at the same value $\mathfrak n \in N$ for $|\gamma|>R$. Therefore $G$ does not depend on $\gamma$ and the supremum on $\gamma$ can be restricted to the set $|\gamma|\leq R$. 

\vspace{0.5em}\noindent
Suppose finally that $\tilde{p}<0$. Then for $\gamma$ big enough, the infimum in $G$ is attained at the maximizer $\overline{\nu}$. For $\gamma$ small enough, the infimum in $G$ is attained at the minimizer $\underline{\nu}$. In both cases, the dependence of $G$ on $\gamma$ is given by the term $\tilde{p} |\gamma|(\sigma(t,x,\overline{\nu})^2 - \sigma(t,x,\underline{\nu})^2 )$ so it follows that $g$ is coercive in $\gamma$.

\vspace{0.5em}\noindent (b) $R$ is continuous as a consequence of the maximum theorem \cite{berge1963topological}. Indeed, the minimizer and maximizer correspondences are upper hemicontinuous and single-valued, therefore they are continuous functions and $R$ is continuous.
\end{proof}


\section{Proof of Theorem \ref{main:Principal}}\label{appendix:Principal} 

The following Lemma is used in the proof of Theorem \ref{main:Principal}. Its proof is omitted, being a path-wise approximation of deterministic Lebesgue integrals. 

\begin{Lemma}\label{key:lemma} Define the process $K(Z,\Gamma)$ by
\begin{equation}K_t(Z,\Gamma) =\int_0^t \left(F^\star(s,Y_s,Z_s,\widehat\sigma^2_s) +\frac12\text{Tr}\left( \widehat\sigma^2_s \Gamma_s\right) - H(s,X_s,Y_s,Z_s,\Gamma_s) \right)ds.
\end{equation}
Then, for any bounded map $\psi$, there exists a sequence $k^p$ of elementary controls such that for any $\varepsilon>0$ and any $p$ big enough
\begin{equation}
\label{inequality:lemma} 
\left| \int_0^t \psi_s dK_s(Z,\Gamma)-\int_0^t \psi_s dk^p_s  \right| \leq \varepsilon, \; \mathcal P_P-q.s. 
\end{equation}
\end{Lemma}

\vspace{0.5em}\noindent
\textbf{Proof of Theorem \ref{main:Principal}}

The proof follows the ideas of \cite{sirbu2014stochastic}. Intuitively, $V_0$ has to be greater than $v^-$ since $v^-$ is roughly speaking the HJBI equation associated with the problem of the Principal when $K$ has the particular decomposition \eqref{decomp:K}. In other words, the value of the unrestricted problem for the Principal has to be a super-solution of such HJBI equation.
\vspace{1em}

\textbf{Step 1. $v^-$ is a viscosity super-solution of \eqref{HJBI}.}

We prove that $v^-$ is a viscosity super-solution of \eqref{HJBI} by contradiction. 

\begin{itemize}

\item[1.] \textbf{The viscosity supersolution property on $[0,T)$}
\begin{itemize}
\item[a.] \textbf{Setting the contradiction.} Let $\varphi$ be some map from $[0,T]\times \mathbb R^d\times \mathbb R\longrightarrow \mathbb R$ continuously differentiable in time and twice continuously differentiable in space. Let $(t_0,x_0,y_0)\in [0,T)\times \mathbb R^d\times \mathbb R$ be such that $v^- - \varphi$ attains a strict local minimum equal to $0$ at this point. We assume (by contradiction) that 
\begin{equation}\label{contradict:supersol}
\partial_t \varphi(t_0,x_0,y_0) + G(t_0,x_0,y_0,\varphi,\nabla_x \varphi,\partial_y \varphi,\Delta_{xx} \varphi ,\partial_{yy}\varphi ,\nabla_{xy}  \varphi) > 0
\end{equation}
In particular, there exists some $(\hat z,\hat \gamma)\in \mathbb R^d\times \mathcal M_{d,d}(\mathbb R)$ and a small $\varepsilon >0$ such that
\begin{align*}
\partial_t \varphi(t_0,x_0,y_0) + \inf_{\mathfrak n  \in N} g (t_0,x_0,y_0,\varphi,\nabla_x \varphi,\partial_y \varphi,\Delta_{xx} \varphi ,\partial_{yy}\varphi ,\nabla_{xy}  \varphi, \hat z ,\hat \gamma, \mathfrak n ) > \varepsilon.
\end{align*}
Recall that $g$ is continuous and $N$ is a compact subset of some finite dimensional space. From Heine's Theorem, we deduce that there exists some $\varepsilon'>0$ such that for any $(t,x,y)\in \mathcal B((t_0,x_0,y_0);\varepsilon')$ we have
\begin{align}\label{ineq:sursol:epsilon}
\partial_t \varphi(t,x,y) + \inf_{\mathfrak n \in N} g (t,x,y,\varphi,\nabla_x \varphi,\partial_y \varphi,\Delta_{xx} \varphi ,\partial_{yy}\varphi ,\nabla_{xy}  \varphi, \hat z ,\hat \gamma, \mathfrak n ) > \varepsilon'.
\end{align}
We denote $\mathcal T_{\varepsilon'}:= \overline{\mathcal B((t_0,x_0,y_0); \varepsilon')}\setminus \mathcal B((t_0,x_0,y_0); \frac{\varepsilon'}{2})$. On $\mathcal T_{\varepsilon'}$, we have $v^->\varphi $ so that the maximum of $\varphi-v^-$ is attained and is negative. Thus, there exists some $\eta>0$ such that $\varphi < v^--\eta$ on $\mathcal T_{\varepsilon'}$. From \cite[Lemma 3.8]{sirbu2014stochastic} there exists a non decreasing sequence $w_n$ in $\mathbb V^-$ converging to $v^-$. Then, there exists $n_0\geq 1$ such that for any $n\geq n_0$ large enough, $\varphi + \frac{\eta}2< w_n$ on $\mathcal T_{\varepsilon'}$. We denote by $w_{n_0+}$ such $w_n$. Thus, for $0<\delta<\frac{\eta}2$ we define
$$ w^\delta:=\begin{cases}
\displaystyle (\varphi +\delta)\vee w_{n_0+}, \text{ on } \mathcal B((t_0,x_0,y_0); \varepsilon'),\\
w_{n_0+}, \text{ outside } \mathcal B((t_0,x_0,y_0); \varepsilon').
\end{cases}
$$
Notice that 
\begin{align}\label{v-:contradiction}
\nonumber w^\delta(t_0,x_0,y_0)&= (\varphi (t_0,x_0,y_0)+\delta)\vee w_{n_0+} (t_0,x_0,y_0)\\
\nonumber&\geq \varphi (t_0,x_0,y_0)+\delta\\
&> v^-(t_0,x_0,y_0).
\end{align}

Thus proving that $w^\delta \in \mathbb V^-$ provides the desired contradiction. From now, we fix some $t\in [0,T]$ and $\tau\in\mathbb{B}^t$. We need to build a strategy $(\tilde Z, \tilde K)\in \mathcal U_{Y_0}(t,\tau)$ such that Property $(ii-)$ in Definition \ref{def:sursous} holds. Recall that $w_{n_0+}\in \mathbb V^-$, thus there exists some elementary strategy $(\tilde Z^1(\tau),\tilde K^1(\tau))\in \mathcal U_{Y_0}(t,\tau)$ such that Property $(ii-)$ in Definition \ref{def:sursous} holds. 

\item[b.] \textbf{Building the elementary strategy and Property (ii-)}
\noindent 
We consider the following strategy that we denote by $(\tilde{\mathcal Z},\tilde{\mathcal K})$.
\begin{itemize}
\item If $\varphi+\delta > w_{n_0+}$ at time $\tau$, we choose the strategy $(\hat z,\hat k^p(\hat z, \hat \gamma))$, where $\hat k^p(\hat z, \hat \gamma)$ is such that inequality \eqref{inequality:lemma} holds with $\frac\varepsilon 2$.
\item Otherwise we follow the elementary strategy $(\tilde Z^1(\tau),\tilde K^1(\tau))$. 
\end{itemize}
Let $\tau_1$ be the first time when $(t,X_t, Y_t)$ exits from $\mathcal B((t_0,x_0,y_0); \varepsilon)$ (which can be $\tau$ itself). On the boundary of this ball, we know that $w^\delta=w_{n_0+}$, thus we choose the strategy $(\tilde Z^1(\tau_1), \tilde K^1(\tau_1))\in \mathcal U(t,\tau_1)$, coinciding with the strategy associated with $w_{n_0+}$ starting at $\tau_1$.

\vspace{0.3em}
Rigorously speaking, define
\begin{align*}
\tilde {\mathcal Z}(s,x(\cdot),y(\cdot))&:=\hat z \mathbf 1_{ \{ \varphi( \tau(x,y), x(\tau(x,y)), y(\tau(x,y))  )+\delta >w_{n_0+} ( \tau(x,y), x(\tau(x,y)), y(\tau(x,y))  ) \} }\\
&+ \tilde Z^1_s(\tau) \mathbf 1_{ \{ \varphi( \tau(x,y), x(\tau(x,y)), y(\tau(x,y))  )+\delta \leq w_{n_0+} ( \tau(x,y), x(\tau(x,y)), y(\tau(x,y))  ) \} },
\end{align*} 
\begin{align*}
\tilde {\mathcal K}(s,x(\cdot),y(\cdot))&:=\hat k_s^p(\hat z,\hat \gamma) \mathbf 1_{ \{ \varphi( \tau(x,y), x(\tau(x,y)), y(\tau(x,y))  )+\delta >w_{n_0+} ( \tau(x,y), x(\tau(x,y)), y(\tau(x,y))  ) \} }\\
&+ \tilde K^1_s(\tau) \mathbf 1_{ \{ \varphi( \tau(x,y), x(\tau(x,y)), y(\tau(x,y))  )+\delta \leq w_{n_0+} ( \tau(x,y), x(\tau(x,y)), y(\tau(x,y))  ) \} },
\end{align*}
and the stopping rule $\tau_1:C([t,T],\mathbb{R}^{d+1}) \longrightarrow [t,T]$ by
$$
\tau_1(x,y) = \underset{\tau(x,y)\leq s \leq T}{\rm inf} ~ (s,x(s),y(s))\in \partial \mathcal B((t_0,x_0,y_0); \varepsilon).
$$
Then, we consider the following strategy
\begin{equation}\label{strategy:proof:sirbu}
\tilde Z:= \tilde{\mathcal Z}\otimes_{\tau_1} \tilde Z^1(\tau_1),\; \tilde K:=\tilde{\mathcal K}\otimes_{\tau_1} \tilde K^1(\tau_1).
\end{equation}
From Lemma 2.8 in \cite{sirbu2014stochastic}, we have $(\tilde Z,\tilde K)\in  \mathcal U(t,\tau).$ It remains to prove that $\tilde K$ satisfies the minimality condition \eqref{eq:minimality condition} to conclude that the strategy defined by \eqref{strategy:proof:sirbu} is in $ \mathcal U_{Y_0}(t,\tau)$. Using a measurability selection argument as in the proof of Theorem 5.3 in \cite{soner2012wellposedness}, for any $\varepsilon>0$ there exists a weak solution $\mathbb P^\varepsilon$ such that $K(\hat z,\hat\gamma)\leq \varepsilon,\; \hat{\mathbb P}^\varepsilon-a.s.$. By Lemma \ref{key:lemma}, we deduce that for any $\varepsilon>0$, any $p$ big enough and any $t\in [0,T]$, $|\hat k_t^p(\hat z, \hat \gamma)|\leq 2\varepsilon,\; \hat{\mathbb P}^\varepsilon-a.s.$ Hence, $(\tilde Z,\tilde K)\in  \mathcal U_{Y_0}(t,\tau).$\vspace{0.3em}

Fix $(Z,K)\in \mathcal U_{Y_0}(t,t)$, $(\mathbb P,\nu)\in \mathcal V_{Y_0}(t,t)$ and $\rho$ a stopping rule in $\mathbb B^t$ such that $\tau\leq \rho\leq T$. With the notations in Definition \ref{def:sursous} (ii-), we define the event
$$A:= \left\{  \varphi(\tau', X_{\tau'},Y_{\tau'})+\delta >w_{n_0+} (\tau', X_{\tau'},Y_{\tau'})\right\}. $$
Applying Ito's formula to $\varphi+\delta$ on $A$, and setting $\sigma_r:=\sigma(r,X_{r}^{\hat z,\hat k^p}, \nu_r)$, 
we get for any $t\leq \tau'\leq s'\leq s\leq \tau_1'$
\begin{align*}
\nonumber \varphi(s,X_s^{\hat z,\hat k^p}, Y_s^{\hat z,\hat k^p})&=\varphi(s',X_{s'}^{\hat z,\hat k^p}, Y_{s'}^{\hat z,\hat k^p})+ \int_{s'}^s \left(\nabla_x \varphi+\partial_y \varphi \hat z\right)\cdot \sigma_r dW_r^\star\\
\nonumber &+\int_{s'}^s \partial_t \varphi +g(r, X_{r}^{\hat z,\hat k^p}, Y_{r}^{\hat z,\hat k^p}, \varphi, \nabla_x\varphi,\partial_y\varphi,\Delta_{xx}\varphi,\partial_{yy}\varphi,\nabla_{xy}\varphi, \hat z,\hat \gamma,\nu_r) dr\\
&+\int_{s'}^s \partial_y\varphi(r, X_{r}^{\hat z,\hat k^p}, Y_{r}^{\hat z,\hat k^p} )\left(d\hat k^p_s - dK_r(\hat z,\hat\gamma)\right).
\end{align*}
From Lemma \ref{key:lemma} together with \eqref{ineq:sursol:epsilon}, we get (for $p$ big enough)
\begin{align*}
\varphi(s,X_s^{\hat z,\hat k^p}, Y_s^{\hat z,\hat k^p})&>\varphi(s',X_{s'}^{\hat z,\hat k^p}, Y_{s'}^{\hat z,\hat k^p})+\int_{s'}^s \left(\nabla_x \varphi+\partial_y \varphi \hat z\right)\cdot \sigma_r dW_r^\star\\
&+(s-s')\frac{\varepsilon'}{2}.
\end{align*}
Thus, $\varphi$ is a sub-martingale on $[\tau,\tau_1]$ under $\mathbb P$ and Property $(ii-)$ is satisfied on $[\tau',\tau'_1]$. On $A^c$, $w_{n_0+}$ automatically satisfies Property $(ii-)$. By noticing that for any $\tau'\leq s\leq \tau_1'$
$$X_s^{t,x,(Z,K)\otimes_{\tau}(\tilde Z,\tilde K),\nu}=\mathbf 1_A X_s^{t,x,(Z,K)\otimes_{\tau}(\hat z,\hat k^p),\nu}+ \mathbf{1}_{A^c}X_s^{t,x,(Z,K)\otimes_{\tau}(\tilde Z^1(\tau),\tilde K^1(\tau)),\nu},  $$
using iterated conditioning and by following the same lines in proof 1.1 of Theorem 3.5 in \cite{sirbu2014stochastic}, we deduce that $w^\delta\in \mathbb V^-$, which contradicts \eqref{v-:contradiction}. Thus, 
$$\partial_t \varphi(t_0,x_0,y_0) + G(t_0,x_0,y_0,\varphi,\nabla_x \varphi,\partial_y \varphi,\Delta_{xx} \varphi ,\partial_{yy}\varphi ,\nabla_{xy}  \varphi) \leq 0.$$
\end{itemize}

\item[2.] \textbf{The viscosity supersolution property at time $T$.} We now aim at proving that $v^-(T,x,y)\geq U_P(L(x)-U_A^{-1}(y))$ for any $(x,y)\in \mathbb R^d\times \mathbb R$. This proof follows the same lines that the step 3 of the proof of Theorem 3.1 in \cite{bayraktar2013stochastic} or the proof of Theorem 3.5, 1.2 in \cite{sirbu2014stochastic}. We assume by contradiction that there exists $(x_0,y_0)\in \mathbb R^d\times \mathbb R$ such that $v^-(T,x_0,y_0)< U_P(L(x_0)-U_A^{-1}(y_0))$. Since $U_P$ is continuous, there exists $\varepsilon>0$ such that 
$$ U_P(L(x)-U_A^{-1}(y))\geq v^-(T,x,y)+\varepsilon,\; (x,y)\in \mathcal B((x_0,y_0);\varepsilon).$$ We define $\mathcal T_\varepsilon:= \overline{\mathcal B((T,x_0,y_0);\varepsilon)} \setminus \mathcal B((T,x_0,y_0);\frac\varepsilon2) .$ Let $\eta>0$ be small enough such that
$$ v^-(T,x_0,y_0) +\varepsilon <\frac{\varepsilon^2}{4\eta} +\inf_{(t,x,y)\in \mathcal T_\varepsilon} v^-(t,x,y).$$
Thus, using exactly the same Dini type arguments that in \cite{sirbu2014stochastic,bayraktar2014dynkin}, there exists $n_0$ big enough such that for some $w_{n_0}\in \mathbb V^-$ we have
$$ v^-(T,x_0,y_0) +\varepsilon <\frac{\varepsilon^2}{4\eta}+ \inf_{(t,x,y)\in \mathcal T_\varepsilon} w_{n_0}(t,x,y).$$
We now define for any $\lambda>0$
$$\varphi^{\varepsilon,\eta,\lambda}(t,x,y):=v^-(T,x_0,y_0)-\frac{\|(x,y)-(x_0,y_0)\|^2}{\eta}-\lambda(T-t).$$
By using the result of Lemma \ref{lemma:coercivity}, for some $\lambda>0$ large enough, we get for any $(t,x,y)\in \overline{\mathcal B((T,x_0,y_0);\varepsilon)}$
$$-\partial_t \varphi^{\varepsilon,\eta,\lambda}(t,x,y)- G(t,x,y,\varphi^{\varepsilon,\eta,\lambda},\nabla_x \varphi^{\varepsilon,\eta,\lambda},\partial_y \varphi^{\varepsilon,\eta,\lambda},\Delta_{xx} \varphi^{\varepsilon,\eta,\lambda} ,\partial_{yy}\varphi^{\varepsilon,\eta,\lambda},\nabla_{xy}  \varphi^{\varepsilon,\eta,\lambda} )<0.$$

Moreover, such $\varphi^{\varepsilon,\eta,\lambda}$ satisfies on $\mathcal T_\varepsilon$
\begin{align*}\varphi^{\varepsilon,\eta,\lambda}(t,x,y)&\leq v^-(T,x_0,y_0) -\frac{\varepsilon^2}{4\eta} \\
&\leq w_{n_0}(t,x,y)-\varepsilon,
\end{align*}
and on $ \mathcal B((x_0,y_0);\varepsilon)$,
$$ \varphi^{\varepsilon,\eta,\lambda}(T,x,y)\leq v^-(T,x,y)\leq U_P(L(x)-U_A^{-1}(y))-\varepsilon.$$
Thus, for $0<\delta<\frac{\eta}2$ we define
$$ w^{\varepsilon, \eta,\lambda,\delta}:=\begin{cases}
\displaystyle (\varphi^{\varepsilon,\eta,\lambda} +\delta)\vee w_{n_0+}, \text{ on } \mathcal B((T,x_0,y_0); \varepsilon),\\
w_{n_0+}, \text{ outside } \mathcal B((T,x_0,y_0); \varepsilon).
\end{cases}
$$

The rest of the proof is analogous similar to the step 1, we show that $ w^{\varepsilon, \eta,\lambda,\delta}\in \mathbb V^-$ and 
\begin{align*}
w^{\varepsilon, \eta,\lambda,\delta}(T,x_0,y_0)&=v^-(T,x_0,y_0)+\delta\\
&>v^-(T,x_0,y_0),
\end{align*}
which leads to a contradiction. We conclude that 
$$v^-(T,x,y)\geq U_P(L(x)-U_A^{-1}(y))$$ for any $(x,y)\in \mathbb R^d\times \mathbb R.$

\end{itemize}

\textbf{Step 2. $v^+$ is a viscosity sub-solution of \eqref{HJBI}.}

We now prove that $v^+$ is a viscosity sub-solution of \eqref{HJBI} by contradiction. 

\begin{itemize}

\item[1.] \textbf{The viscosity subsolution property on $[0,T)$}
\begin{itemize}
\item[a.] \textbf{Setting the contradiction.} Let $\varphi$ be some map from $[0,T]\times \mathbb R^d\times \mathbb R\longrightarrow \mathbb R$ continuously differentiable in time and twice continuously differentiable in space. Let $(t_0,x_0,y_0)\in [0,T)\times \mathbb R^d\times \mathbb R$ be such that $v^+ - \varphi$ attains a strict local maximum equal to $0$ at this point. We assume (by contradiction) that 
\begin{equation}\label{contradict:subsol}
\partial_t \varphi(t_0,x_0,y_0) + G(t_0,x_0,y_0,\varphi,\nabla_x \varphi,\partial_y \varphi,\Delta_{xx} \varphi ,\partial_{yy}\varphi ,\nabla_{xy}  \varphi) < 0
\end{equation}
Then, for any $(z,\gamma)\in \mathbb R^d\times \mathcal M_{d,d}(\mathbb R)$, we have 

\begin{equation*}
\partial_t \varphi(t_0,x_0,y_0) + \inf_{\mathfrak n \in N} g(t_0,x_0,y_0,\varphi,\nabla_x \varphi,\partial_y \varphi,\Delta_{xx} \varphi ,\partial_{yy}\varphi ,\nabla_{xy}  \varphi,z,\gamma,\mathfrak n ) < 0.
\end{equation*}
Therefore, there exists $\varepsilon>0$ and $\hat\nu(z,\gamma)\in N$ such that
\begin{equation*}
\partial_t \varphi(t_0,x_0,y_0) + g(t_0,x_0,y_0,\varphi,\nabla_x \varphi,\partial_y \varphi,\Delta_{xx} \varphi ,\partial_{yy}\varphi ,\nabla_{xy}  \varphi,z,\gamma,\hat\nu(z,\gamma)) < -\varepsilon.
\end{equation*}
Using the continuity of our applications, we deduce that on $\mathcal B((t_0,x_0,y_0);\varepsilon)$ we have, for a small $\varepsilon'>0$,
\begin{equation*}
\partial_t \varphi(t,x,y) + g(t,x,y,\varphi,\nabla_x \varphi,\partial_y \varphi,\Delta_{xx} \varphi ,\partial_{yy}\varphi ,\nabla_{xy}  \varphi,z,\gamma,\hat\nu(z,\gamma)) < -\varepsilon'.
\end{equation*}

We denote $\mathcal T_{\varepsilon'}:= \overline{\mathcal B((t_0,x_0,y_0); \varepsilon')}\setminus \mathcal B((t_0,x_0,y_0); \frac{\varepsilon'}{2})$. On $\mathcal T_{\varepsilon'}$, we have $v^+<\varphi $ so that the minimum of $\varphi-v^+$ is attained and is positive. Thus, there exists some $\eta>0$ such that $\varphi > v^++\eta$ on $\mathcal T_{\varepsilon'}$. From \cite[Lemma 3.8]{sirbu2014stochastic} there exists a non increasing sequence $w_n$ in $\mathbb V^+$ converging to $v^+$. Then, there exists a positive real $n_0\geq 1$ such that for any $n\geq n_0$ large enough, $\varphi - \frac{\eta}2> w_n$ on $\mathcal T_{\varepsilon'}$. We denote by $w_{n_0+}$ such $w_n$. Thus, for $0<\delta<\frac{\eta}2$ we define
$$ w^\delta:=\begin{cases}
\displaystyle (\varphi -\delta)\wedge w_{n_0+}, \text{ on } \mathcal B((t_0,x_0,y_0); \varepsilon'),\\
w_{n_0+}, \text{ outside } \mathcal B((t_0,x_0,y_0); \varepsilon').
\end{cases}
$$

Notice that 
\begin{align}\label{v+:contradiction}
\nonumber w^\delta(t_0,x_0,y_0)&= (\varphi (t_0,x_0,y_0)-\delta)\wedge w_{n_0+} (t_0,x_0,y_0)\\
\nonumber&\leq \varphi (t_0,x_0,y_0)-\delta\\
&<v^+(t_0,x_0,y_0).
\end{align}

Thus proving that $w^\delta \in \mathbb V^+$ provides the desired contradiction. From now, we fix $t\in [0,T]$, a stopping rule $\tau\in\mathbb{B}^t$ and $(Z, K)\in \mathcal U_{Y_0}(t,\tau)$. We need to build a strategy $(\mathbb P, \tilde \nu)\in \mathcal V_{Y_0}(t,\tau)$ such that Property $(ii+)$ in the definition \ref{def:sursous} of a super-solution holds. Recall that $w_{n_0+}\in \mathbb V^+$, thus for the fixed $(Z, K)\in \mathcal U_{Y_0}(t,\tau)$, there exists some elementary strategy $(\tilde{\mathbb P}, \tilde{\nu}^1)\in \mathcal V_{Y_0}(t,\tau)$ such that Property $(ii+)$ in Definition \ref{def:sursous} holds.

\item[b.] \textbf{Building the elementary strategy and Property (ii+)}
\noindent 
We consider the following strategy that we denote by $\tilde \nu$.

\begin{itemize}
\item If $\varphi-\delta < w_{n_0+}$ at time $\tau$, we choose the strategy $( \hat {\mathbb P},\hat{\nu}(Z,0))$, where $\hat{ \mathbb P} \in \mathcal P_A$ is such that the minimality condition $(ii)$ in Theorem \ref{thm:optimaleffort} holds with control $K$, 
\item Otherwise we follow the elementary strategy $(\tilde{\mathbb P},\tilde{\nu}^1)$. 
\end{itemize}
 
The rest of this part is completely similar to Step 1., paragraph 1.b. with control
\begin{align*}
\tilde {\nu}_t&:=\hat \nu(Z,0) \mathbf 1_{\{ \varphi-\delta <w_{n_0+} \} }+ \tilde{\nu}^1_t \mathbf 1_{ \{ \varphi-\delta \geq w_{n_0+} \} },
\end{align*}
and considering the event
$$
\tilde A:= \left\{  \varphi(\tau', X_{\tau'},Y_{\tau'})-\delta <w_{n_0+} (\tau', X_{\tau'},Y_{\tau'})\right\}. 
$$
Applying Ito's formula to $\varphi+\delta$ on $\tilde A$, and setting $\sigma_r:=\sigma(r,X_{r}^{\tilde \nu}, \hat \nu(Z,0))$, 
we get for any $t\leq \tau'\leq s'\leq s\leq \tau_1'$
\begin{align*}
\nonumber \varphi(s,X_s^{\tilde \nu}, Y_s^{\tilde \nu})&=\varphi(s',X_{s'}^{\tilde \nu}, Y_{s'}^{\tilde \nu})+ \int_{s'}^s \left(\nabla_x \varphi+\partial_y \varphi Z \right)\cdot \sigma_r dW_r^\star +\int_{s'}^s \partial_t \varphi dr\\
\nonumber & +g(r, X_{r}^{\tilde \nu}, Y_{r}^{\tilde \nu}, \varphi, \nabla_x\varphi,\partial_y\varphi,\Delta_{xx}\varphi,\partial_{yy}\varphi,\nabla_{xy}\varphi, Z,0,\hat \nu(Z,0)) dr\\
&+\int_{s'}^s \partial_y\varphi(r, X_{r}^{\tilde \nu}, Y_{r}^{\tilde \nu} )  dK_s.
\end{align*}
Since $K=0$ under $\hat{\mathbb P}$ we have that $\varphi$ is a super--martingale under $\hat{\mathbb P}$ and $(ii+)$ is satisfied on $[\tau,\tau_1]$. We thus deduce similarly that $w^\delta\in \mathbb V^+$ which contradicts \eqref{v+:contradiction} so that for any $(t,x,y)\in [0,T)\times \mathbb R^d\times\mathbb R $
$$-\partial_t \varphi(t,x,y) - G(t,x,y,\varphi,\nabla_x \varphi,\partial_y \varphi,\Delta_{xx} \varphi ,\partial_{yy}\varphi ,\nabla_{xy}  \varphi) \leq 0$$
\end{itemize}

\item[2.] \textbf{The viscosity supersolution property at time $T$.}
We now aim at proving that $v^+(T,x,y)\leq U_P(L(x)-U_A^{-1}(y))$ for any $(x,y)\in \mathbb R^d\times \mathbb R$. This proof follows the same lines that the previous step 1.2. We assume by contradiction that there exists $(x_0,y_0)\in \mathbb R^d\times \mathbb R$ such that $v^+(T,x_0,y_0)> U_P(L(x_0)-U_A^{-1}(y_0))$. By continuity, there exists $\varepsilon>0$ such that 
$$ U_P(L(x)-U_A^{-1}(y))\leq v^+(T,x,y)-\varepsilon,\; (x,y)\in \mathcal B((x_0,y_0);\varepsilon).$$ We define $\mathcal T_\varepsilon:= \overline{\mathcal B((T,x_0,y_0);\varepsilon)} \setminus \mathcal B((T,x_0,y_0);\frac\varepsilon2) .$ Let $\eta>0$ be small enough such that
$$ v^+(T,x_0,y_0)+\frac{\varepsilon^2+4\ln(1+\frac\varepsilon 2)}{4\eta}  >\varepsilon+\sup_{(t,x,y)\in \mathcal T_\varepsilon} v^+(t,x,y).$$
Thus, using exactly the same Dini type arguments that in \cite{sirbu2014stochastic,bayraktar2014dynkin}, there exists $n_0$ big enough such that for some $w_{n_0}\in \mathbb V^+$ we have
$$v^+(T,x_0,y_0)+\frac{\varepsilon^2+4\ln(1+\frac\varepsilon 2)}{4\eta}  >\varepsilon+ \sup_{(t,x,y)\in \mathcal T_\varepsilon} w_{n_0}(t,x,y).$$
We now define for any $\lambda>0$
$$\varphi^{\varepsilon,\eta,\lambda}(t,x,y):=v^+(T,x_0,y_0)+\frac{\|x-x_0\|^2+\ln(1+|y-y_0|)}{\eta}+\lambda(T-t).$$
By using the result of Lemma \ref{lemma:coercivity}, for some $\lambda>0$ large enough, we get for any $(t,x,y)\in \overline{\mathcal B((T,x_0,y_0);\varepsilon)}$
$$-\partial_t \varphi^{\varepsilon,\eta,\lambda}(t,x,y)- G(t,x,y,\varphi^{\varepsilon,\eta,\lambda},\nabla_x \varphi^{\varepsilon,\eta,\lambda},\partial_y \varphi^{\varepsilon,\eta,\lambda},\Delta_{xx} \varphi^{\varepsilon,\eta,\lambda} ,\partial_{yy}\varphi^{\varepsilon,\eta,\lambda},\nabla_{xy}  \varphi^{\varepsilon,\eta,\lambda} )>0.$$

In particular, such $\varphi^{\varepsilon,\eta,\lambda}$ satisfies on $\mathcal T_\varepsilon$
\begin{align*}\varphi^{\varepsilon,\eta,\lambda}(t,x,y)&\geq v^+(T,x_0,y_0)+\frac{\varepsilon^2+4\ln(1+\frac\varepsilon 2)}{4\eta}  \\
&\geq w_{n_0}(t,x,y)+\varepsilon,
\end{align*}
and on $ \mathcal B((x_0,y_0);\varepsilon)$,
$$ \varphi^{\varepsilon,\eta,\lambda}(T,x,y)\geq v^-(T,x,y)\geq U_P(L(x)-U_A^{-1}(y))+\varepsilon.$$
Thus, for $0<\delta<\varepsilon$ small enough we define
$$ w^{\varepsilon, \eta,\lambda,\delta}:=\begin{cases}
\displaystyle (\varphi^{\varepsilon,\eta,\lambda} -\delta)\wedge w_{n_0+}, \text{ on } \mathcal B((T,x_0,y_0); \varepsilon),\\
w_{n_0+}, \text{ outside } \mathcal B((T,x_0,y_0); \varepsilon).
\end{cases}
$$

The rest of the proof is completely similar to the step 1.b, we show that $ w^{\varepsilon, \eta,\lambda,\delta}\in \mathbb V^+$ and 
\begin{align*}
w^{\varepsilon, \eta,\lambda,\delta}(T,x_0,y_0)&=v^+(T,x_0,y_0)+\delta\\
&>v^+(T,x_0,y_0),
\end{align*}
which leads to a contradiction. We deduce that 
$$v^+(T,x,y)\leq U_P(L(x)-U_A^{-1}(y))$$ for any $(x,y)\in \mathbb R^d\times \mathbb R.$

\end{itemize}

\textbf{General conclusion.}
In step 1 (resp. in step 2) we have proved that $v^-$ is a viscosity super-solution (resp. $v^+$ is a viscosity sub-solution) of the HJBI equation \eqref{HJBI}. If a comparison theorem in the viscosity sense holds, then we deduce from \eqref{ineg:v+-V} that 
$$v^-(0,x,Y_0)\leq  V_0^P(Y_0) \leq v^+(0,x,Y_0) \leq v^-(0,x,Y_0),\; x\in \mathbb R^d,$$
which proves the theorem.

\end{document}